\documentclass[10pt]{amsart}
\pdfoutput=1
\usepackage[
paper=a4paper,
text={147mm,230mm},centering
]{geometry}
\iftrue
\makeatletter
\def\@settitle{%
  \vspace*{-20pt}
  \begin{flushleft}%
    \baselineskip14\p@\relax
    \normalfont\bfseries\LARGE
    \@title
  \end{flushleft}%
}
\def\@setauthors{%
  \begingroup
  \def\thanks{\protect\thanks@warning}%
  \trivlist
  \large \@topsep30\p@\relax
  \advance\@topsep by -\baselineskip
  \item\relax
  \author@andify\authors
  \def\\{\protect\linebreak}%
  \authors
  \ifx\@empty\contribs
  \else
    ,\penalty-3 \space \@setcontribs
    \@closetoccontribs
  \fi
  \normalfont
  \@setaddresses
  \endtrivlist
  \endgroup
}
\def\@setaddresses{\par
  \nobreak \begingroup\raggedright
  \small
  \def\author##1{\nobreak\addvspace\smallskipamount}%
  \def\\{\unskip, \ignorespaces}%
  \interlinepenalty\@M
  \def\address##1##2{\begingroup
    \par\addvspace\bigskipamount\noindent
    \@ifnotempty{##1}{(\ignorespaces##1\unskip) }%
    {\ignorespaces##2}\par\endgroup}%
  \def\curraddr##1##2{\begingroup
    \@ifnotempty{##2}{\nobreak\noindent\curraddrname
      \@ifnotempty{##1}{, \ignorespaces##1\unskip}\/:\space
      ##2\par}\endgroup}%
  \def\email##1##2{\begingroup
    \@ifnotempty{##2}{\smallskip\nobreak\noindent E-mail address%
      \@ifnotempty{##1}{, \ignorespaces##1\unskip}\/:\space
      \ttfamily##2\par}\endgroup}%
  \def\urladdr##1##2{\begingroup
    \def~{\char`\~}%
    \@ifnotempty{##2}{\nobreak\noindent\urladdrname
      \@ifnotempty{##1}{, \ignorespaces##1\unskip}\/:\space
      \ttfamily##2\par}\endgroup}%
  \addresses
  \endgroup
  \global\let\addresses=\@empty
}
\def\@setabstracta{%
    \ifvoid\abstractbox
  \else
    \skip@25\p@ \advance\skip@-\lastskip
    \advance\skip@-\baselineskip \vskip\skip@
    \box\abstractbox
    \prevdepth\z@ % because \abstractbox is a vtop
    \vskip-10pt
  \fi
}
\renewenvironment{abstract}{%
  \ifx\maketitle\relax
    \ClassWarning{\@classname}{Abstract should precede
      \protect\maketitle\space in AMS document classes; reported}%
  \fi
  \global\setbox\abstractbox=\vtop \bgroup
    \normalfont\small
    \list{}{\labelwidth\z@
      \leftmargin0pc \rightmargin\leftmargin
      \listparindent\normalparindent \itemindent\z@
      \parsep\z@ \@plus\p@
      
    }%
    \item[\hskip\labelsep\bfseries\abstractname.]%
}{%
  \endlist\egroup
  \ifx\@setabstract\relax \@setabstracta \fi
}
\def\section{\@startsection{section}{1}%
  \z@{-1.2\linespacing\@plus-.5\linespacing}{.8\linespacing}%
  {\normalfont\bfseries\large}}
\def\subsection{\@startsection{subsection}{2}%
  \z@{-.8\linespacing\@plus-.3\linespacing}{.3\linespacing\@plus.2\linespacing}%
  {\normalfont\bfseries}}
\def\subsubsection{\@startsection{subsubsection}{3}%
  \z@{.7\linespacing\@plus.1\linespacing}{-1.5ex}%
  {\normalfont\itshape}}
\def\@secnumfont{\bfseries}
\makeatother
\fi %\iftrue/false for amsart.cls hack

\usepackage{amssymb}
\usepackage[all]{xy}
\usepackage{graphicx,color,float}
\usepackage[bookmarks]{hyperref}
\usepackage{enumitem}
\usepackage[textwidth=1.3in,color=yellow]{todonotes}
\usepackage{amsmath}
\usepackage{pb-diagram,pb-xy}
\usepackage{url}
\usepackage{tikz}

\textwidth=\paperwidth \advance\textwidth by-2\oddsidemargin
\advance\oddsidemargin by-1in
\evensidemargin=\oddsidemargin
\calclayout

\def\Z{\mathbb{Z}}

\def\Q{\mathbb{Q}}

\def\d{\partial}

\def\tilde{\widetilde}

\def\Hom{\operatorname{Hom}}

\def\+{\oplus}
\def\hat{\widehat}

\newcommand{\da}{\ar@{-->}}
\newcommand{\dar}{\ar@{.>}}
\newcommand{\lar}{\ar@{-}}

\newcommand{\tc}{\textcolor}
\numberwithin{equation}{section}
\theoremstyle{plain}
\newtheorem{theorem}{Theorem}[section]
\newtheorem{proposition}[theorem]{Proposition}

\newtheorem{lemma}[theorem]{Lemma}

\newtheorem{theoremalpha}{Theorem}
 
\theoremstyle{definition}
\newtheorem{definition}[theorem]{Definition}
\newtheorem{example}[theorem]{Example}
\newtheorem{remark}[theorem]{Remark}

\newtheorem*{acknowledgement}{Acknowledgement}
\def\to{\mathchoice{\longrightarrow}{\rightarrow}{\rightarrow}{\rightarrow}}
\makeatletter
\newcommand{\shortxra}[2][]{\ext@arrow 0359\rightarrowfill@{#1}{#2}}
\def\longrightarrowfill@{\arrowfill@\relbar\relbar\longrightarrow}
\newcommand{\longxra}[2][]{\ext@arrow 0359\longrightarrowfill@{#1}{#2}}
\renewcommand{\xrightarrow}[2][]{\mathchoice{\longxra[#1]{#2}}%
  {\shortxra[#1]{#2}}{\shortxra[#1]{#2}}{\shortxra[#1]{#2}}}
\makeatother

\begin{document}

\title[Links with nontrivial Alexander polynomial]
{Links with nontrivial Alexander polynomial which are topologically concordant to the Hopf link}

\author{Min Hoon Kim}
\address{
  School of Mathematics\\
  Korea Institute for Advanced Study \\
  Seoul 02455\\
  Republic of Korea
}
\email{kminhoon@kias.re.kr}

\author{David Krcatovich}
\address{
	Rice University\\ 
	Department of Mathematics\\ 
	6100 Main Street, Houston, TX 77005\\
	USA}
\email{dk27@rice.edu}
\author{JungHwan Park}
\address{
	Rice University\\ 
	Department of Mathematics\\ 
	6100 Main Street, Houston, TX 77005\\
	USA}
\email{jp35@rice.edu}

% Instead of \subjclass[2000]{...}, 
% I put a workaround for the old versions of amsart.
%\def\subjclassname{\textup{2000} Mathematics Subject Classification}
%\expandafter\let\csname subjclassname@1991\endcsname=\subjclassname
%\expandafter\let\csname subjclassname@2000\endcsname=\subjclassname
%\subjclass{Primary 20J05, 57M07, Secondary 55P60, 57M27}

%\keywords{Knot concordance, Cochran-Orr-Teichner filtration, Grope, Whitney tower}

\maketitle
\begin{abstract}We give infinitely many 2-component links with unknotted components which are topologically concordant to the Hopf link, but not smoothly concordant to any 2-component link with trivial Alexander polynomial. Our examples are pairwise non-concordant.\end{abstract}

\section{Introduction}\label{section:Introduction}
%=================================================================================================================================
Freedman's topological 4-dimensional surgery theory \cite{Freedman:1982-1} has a well-known consequence that knots with trivial Alexander polynomial are topologically slice (see \cite{Freedman:1982-2,Freedman-Quinn:1990-1,Garoufalidis-Teichner:2004-1}). Inspired by the result of Freedman, Hillman  \cite[Section 7.6]{Hillman:2002-1}  proposed a surgery program for 2-component links with linking number one. By completing the program of Hillman, Davis \cite{Davis:2006-1} proved that every 2-component link with linking number 1 which has trivial Alexander polynomial is topologically concordant to the Hopf link. In other words, the Alexander polynomials of knots and links determine their topological concordance type in these cases. Interestingly, Cha, Friedl and Powell \cite{Cha-Friedl-Powell:2014-1} proved that these two cases are exceptional. Namely, they showed that the link concordance class is not determined by the Alexander polynomial in any other cases.

 Based on Donaldson's diagonalization theorem, Casson and Akbulut observed that there are knots with trivial Alexander polynomial which are not smoothly slice (their results are unpublished, see Cochran and Gompf \cite{Cochran-Gompf:1988-1}). Following this result, smooth concordance of topologically slice knots has been studied extensively using various modern techniques including gauge theory, Heegaard Floer homology and Khovanov homology (for example, see \cite{Gompf:1986-1,Endo:1995-1,Manolescu-Owens:2007-1,Livingston:2008-1,Hedden-Livingston-Ruberman:2012-1,Cochran-Harvey-Horn:2013-1,Cochran-Horn:2015-1,Hom:2015-3,Ozsvath-Stipsicz-Szabo:2014-1,Hedden-Kim-Livingston:2016-1,Donald-Vafaee:2016-1}). Most of these examples  have trivial Alexander polynomial. It was natural to ask whether every topologically slice knot is smoothly concordant to a knot with trivial Alexander polynomial. Hedden, Livingston and Ruberman \cite{Hedden-Livingston-Ruberman:2012-1} constructed infinitely many topologically slice knots which are not smoothly concordant to any knot with trivial Alexander polynomial. Actually, they showed that their examples are linearly independent in the smooth knot concordance group.

Cha, T.\ Kim, Ruberman and Strle \cite{Cha-Kim-Ruberman-Strle:2010-1} constructed an infinite family of links with unknotted components which are topologically concordant, but not smoothly concordant, to the Hopf link. By studying satellite operators, Davis and Ray \cite{Davis-Ray:2017-1} constructed another family of  links with the same properties. These families of links have trivial Alexander polynomial. (This fact can be checked using C-complexes.) Inspired by the result of Hedden-Livingston-Ruberman \cite{Hedden-Livingston-Ruberman:2012-1}, it is natural to ask whether there are links with unknotted components which are topologically concordant to the Hopf link but not smoothly concordant to any link with trivial Alexander polynomial. The goal of this paper is to answer that question.

\begin{theoremalpha}\label{theorem:A}There exist infinitely many, pairwise non-concordant $2$-component links with unknotted components which are topologically concordant to the Hopf link, but not smoothly concordant to any $2$-component link with trivial Alexander polynomial.
\end{theoremalpha}
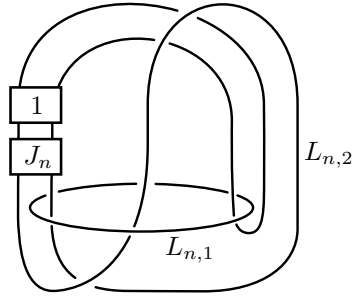
\begin{figure}[htb]\label{figure:mainexample}
\centering
\includegraphics[width=1.5in]{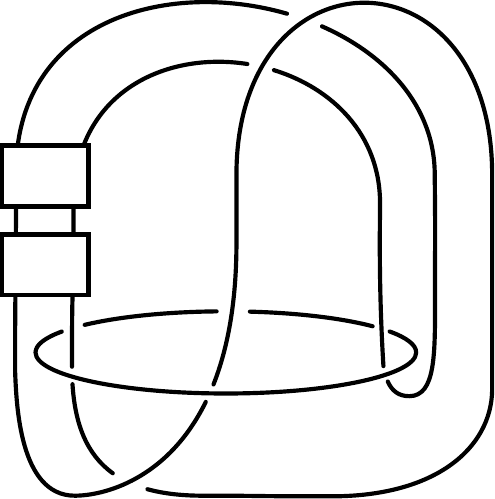}
\put(-100.5,67.5){$1$}
\put(-103,48.4){$J_n$}
\put(-50,14){$L_{n,1}$}
\put(2,50){$L_{n,2}$}
\caption{A family of two component links $L_n=L_{n,1}\sqcup L_{n,2}$.}
\end{figure}

Our family of examples $L_n$ is given in Figure \ref{figure:mainexample}. It is immediate to see that the components of $L_n$ are unknotted. Here the knot $J_n$ is 
\[J_n=(nD)_{2,4n-1} \# -T_{2,4n-1} \#2(n-1)D\]
where $D$ is the (positive) Whitehead double of the right handed trefoil. We will see that $J_n$ is topologically slice in Lemma \ref{lemma:Jn}. Assuming that $J_n$ is topologically slice, Figure~\ref{figure:concordancehopf} shows that $L_n$ is topologically concordant to the Hopf link for any $n$.

\begin{figure}[htb]
\centering
\includegraphics[width=5.5in]{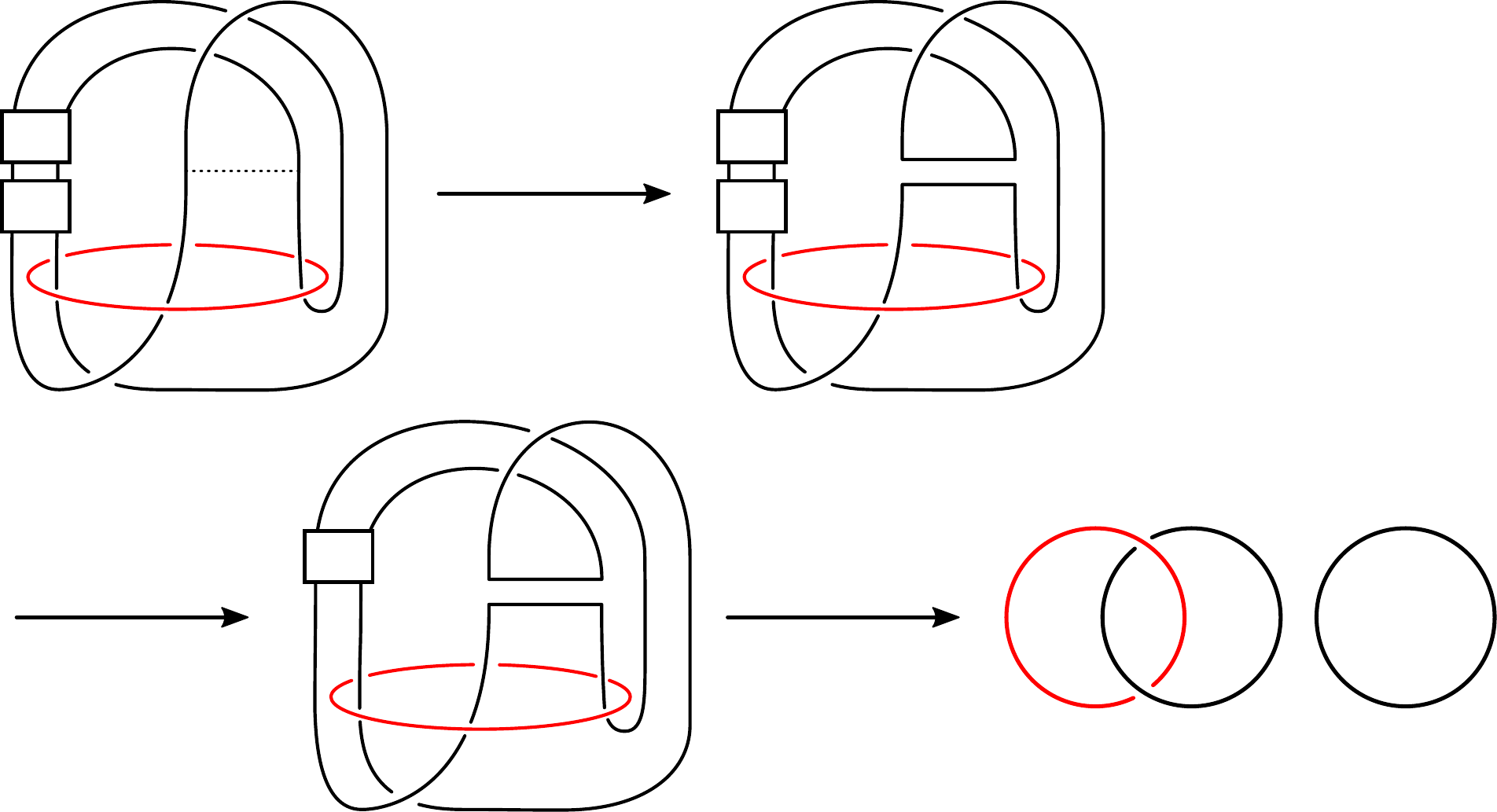}
\put(-390,176){$1$}
\put(-393,158){$J_n$}
\put(-200,176){$1$}
\put(-202.5,158){$J_n$}
\put(-310,64.5){$1$}
\put(-277,171.5){band move}
\put(-387,73){topological}
\put(-390,60){concordance}
\put(-191,60){isotopy}
\caption{A proof that the link $L_n$ is topologically concordant to the Hopf link. The first figure represents $L_n$.  The second figure is obtained by a band move from the first figure. (The band is depicted as the dotted line in the first figure.) The link in the third figure is topologically concordant to the link in the second figure since $J_n$ is topologically slice. It is straightforward to see that the third figure is isotopic to the split union of the Hopf link and the unknot.}
\label{figure:concordancehopf}
\end{figure}
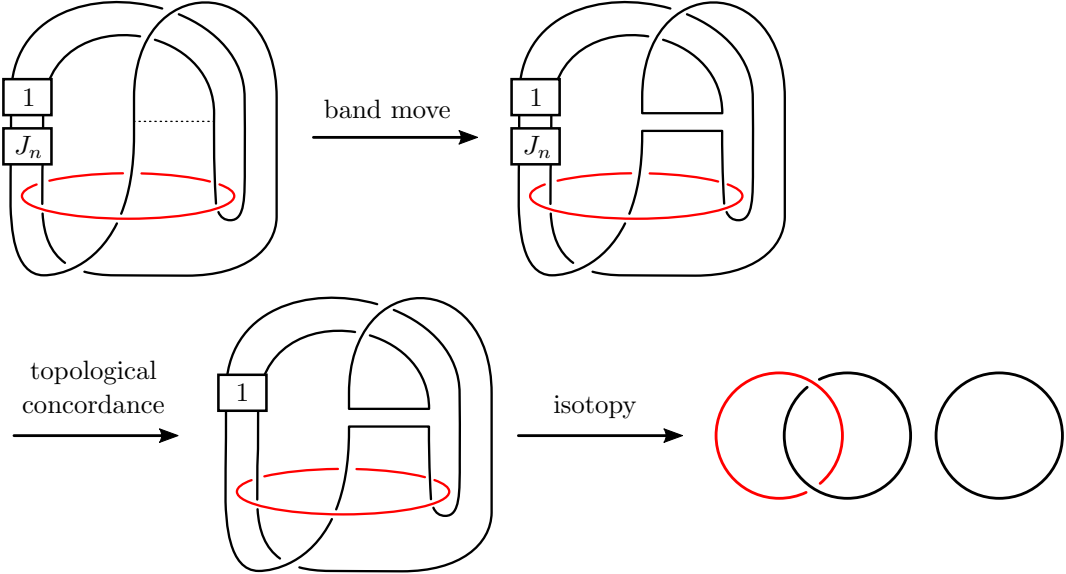
The difficult part of Theorem \ref{theorem:A} is to prove $L_n$ is not smoothly concordant to any $2$-component link with trivial Alexander polynomial. This part has two ingredients. The first ingredient is the following vanishing theorem which is analogous to an obstruction introduced in \cite{Hedden-Livingston-Ruberman:2012-1}.

\begin{theoremalpha}\label{theorem:B} Suppose that $L=L_1\sqcup L_2$ is a $2$-component link with linking number $1$. Let $Y_L$ be the $3$-manifold obtained by doing $1$-surgery on $S^3$ along $L_1$. Denote the image of $L_2$ in $Y_L$ by $K_L$. Let $\Sigma_{K_L}$ be the $2$-fold cover of $Y_L$ branched along $K_L$. If $L$ is smoothly concordant to a link $J$ with $\Delta_J(s,t)=1$, then the Hedden-Livingston-Ruberman obstruction vanishes for $\Sigma_{K_L}$.
\end{theoremalpha}
For the precise statement of vanishing Hedden-Livingston-Ruberman obstruction, see Theorem~\ref{theorem:mainobstruction}. To prove that $\Sigma_{K_{L_n}}$ has non-vanishing Hedden-Livingston-Ruberman obstruction, we will need the following theorem which is based on results of \cite{Kim-Park:2016-1} and the reduced Floer chain complex introduced by the second author \cite{Krcatovich:2015-1}. Here, $\{V_k(K)\}$ is a sequence of smooth concordance invariants of~$K$ which was introduced by Rasmussen \cite{Rasmussen:2003-1} and then further studied by Ni and Wu \cite{Ni-Wu:2015-1}. Moreover, Ni and Wu showed that correction terms of $S^3_{p/q}(K)$ can be computed using the sequence $\{V_k(K)\}$. For more details, see Section~\ref{section:visequence}.
\begin{theoremalpha}\label{theorem:C} Let $J_n$ be the knot $(nD)_{2,4n-1} \# -T_{2,4n-1} \#2(n-1)D$ where $D$ is the Whitehead double of the right handed trefoil. For any integer $n\geq 1$, the knot $J_n$
is topologically slice and satisfies $V_0(2J_n)=2n$ and $V_1(2J_n)=2n-1$.
\end{theoremalpha}

The rest of the paper is organized as follows. In Section \ref{section:HLR}, we will show that the Alexander polynomial of $K_L$ is determined by the Alexander polynomial of $L$. We will recall the Hedden-Livingston-Ruberman obstruction and prove Theorem \ref{theorem:B}. In Section \ref{section:d-invariants}, we will recall necessary results on Heegaard Floer invariants of knots and prove Theorem \ref{theorem:C}. In Section
~\ref{section:maintheorem}, we prove Theorem \ref{theorem:A}.

\begin{acknowledgement}This paper was partially written when the authors participated in the conference on knot concordance and 4-manifolds which was held at the Max-Planck-Institut f\"{u}r Mathematik in October 17--21, 2016. The authors would like to thank the Max-Planck-Institut f\"{u}r Mathematik for its stimulating working environment. The first author was partially supported by the Overseas Research Program for Young Scientists through the Korea Institute for Advanced Study. The third author was partially supported by the National Science Foundation grant DMS-1309081. The first and the third authors thank the Hausdorff Institute for Mathematics in Bonn for both support and its outstanding research environment. The authors thank Jae Choon Cha, Matthew Hedden and Daniel Ruberman for comments on the first version of this paper. Finally, we are grateful to the anonymous referee for the detailed and thoughtful suggestions.
\end{acknowledgement}

%==============================================================================================================================
\section{A link analogue of the Hedden-Livingston-Ruberman obstruction}\label{section:HLR}
\subsection{Alexander polynomials of knots in homology 3-spheres}\label{section:blow-down}

%==============================================================================================================================

We first recall necessary definitions of the Alexander polynomials of knots and links following \cite[Chapter~7]{Kawauchi:1996-1} and \cite[Chapter~3]{Hillman:2002-1}. Let $\Lambda$ be a ring which is a unique factorization domain and N\"{o}therian. (In our cases, $\Lambda$ will be either $\Z[t^{\pm 1}]$ or $\Z[s^{\pm 1},t^{\pm 1}]$.) Let $M$ be a finitely generated $\Lambda$-module. Since $M$ is finitely generated, there is an epimorphism $\varphi\colon \Lambda^n\to M$ for some positive integer~$n$. Since $\Lambda$ is N\"{o}therian, we can assume that  the kernel of $\varphi$ is generated by $m$ elements with $m\geq n$. There is an $m\times n$ matrix $P$ over $\Lambda$ and an exact sequence 
\[\Lambda^m\xrightarrow{P} \Lambda^n\to M\to 0.\]
The matrix $P$ is called a \emph{presentation matrix} of $M$. The \emph{$0$-th characteristic polynomial} of $M$, denoted by $\Delta_0(M)\in \Lambda$, is the greatest common divisor of the elements of the ideal generated by $n\times n$ minors of $P$. It is known that $\Delta_0(M)$ is well-defined up to multiplication by a unit of $\Lambda$. (That is, $\Delta_0(M)$ does not depend on the choice of a presentation matrix $P$.) 
\begin{remark}\label{remark:shortexact} For a given short exact sequence of finitely generated $\Lambda$-modules,
\[0\to M_1\to M\to M_2\to 0\]
we have $\Delta_0(M)\overset{\cdot}=\Delta_0(M_1)\Delta_0(M_2)$ where $\overset{\cdot}=$ denotes the equality up to multiplication by a unit of $\Lambda$ (see \cite[Lemma 7.2.7]{Kawauchi:1996-1}).
\end{remark}
For an oriented knot $K$ in a homology 3-sphere $Y$, let $E_K=Y-\nu(K)$. Let $\tilde{E_K}\to E_K$ be the infinite cyclic cover induced by the abelianization map $ \pi_1(E_K)\to H_1(E_K)$. 
As in \cite[Chapter 2]{Hillman:2002-1}, $H_*(E_K;\Z[t^{\pm 1}])$ denotes the homology of $\tilde{E_K}$. We define the Alexander polynomial $\Delta_K(t)$ of $K$ to be $\Delta_0(H_1(E_K;\Z[t^{\pm 1}]))$.

Similarly, for a 2-component link $L$ in $S^3$, let $E_L=S^3-\nu(L)$. Let $\tilde{E_L}\to E_L$ be the $\Z\oplus \Z$-cover induced by the abelianization map $\pi_1(E_L)\to H_1(E_L)$. We define the Alexander polynomial $\Delta_L(s,t)$ of $L$ to be $\Delta_0(H_1(E_L;\Z[s^{\pm 1},t^{\pm 1}]))$.

Technically, our definitions of the Alexander polynomials of knots/links seem to be different from the ones given in \cite{Kawauchi:1996-1}, but they are in fact equivalent (see \cite[Proposition 7.3.4(2)]{Kawauchi:1996-1}).

\begin{definition}[Knot obtained by 1-surgery on the first component]\label{definition:knotobtainedbysurgery} Let $L=L_1\sqcup L_2$ be a 2-component link in $S^3$ with linking number 1. Let $Y_L$ be the homology 3-sphere obtained from $S^3$ by doing 1-surgery along $L_1$. Let $K_L$ be the knot which is the image of $L_2$ in $Y_L$. We say $K_L$ is \emph{the knot obtained from $L$ by doing $1$-surgery on the first component}.
\end{definition}
For example, the knot $K_{L_n}$ drawn in Figure \ref{figure:blowdown} is the knot obtained from $L_n$ by doing 1-surgery on the first component. A simple proof of the following lemma was suggested by the anonymous referee.
\begin{figure}[h]
\centering
\includegraphics[width=1.5in]{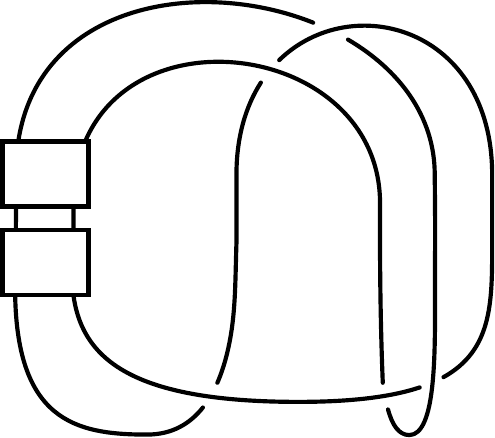}
\put(-100,54){$1$}
\put(-103,36){$J_n$}
\caption{A knot $K_{L_n}$ obtained from $L_n$ by doing $1$-surgery on the first component.}
\label{figure:blowdown}
\end{figure}
\begin{lemma}\label{lemma:Alexofblowdown}Let $L$ be a $2$-component link with linking number $1$. If the Alexander polynomial of $L$ is trivial, then the Alexander polynomial of the knot $K_L$ is also trivial.
\end{lemma}
\begin{proof}Suppose that $L=L_1\sqcup L_2$ is a $2$-component link with linking number $1$, and has trivial Alexander polynomial. Let $Y_L$ be the homology 3-sphere obtained from $S^3$ by doing 1-surgery along $L_1$. Consider the link $L'\subset Y_L$ whose first component is the core of the $1$-surgery on $L_1$ and second component is the image of $L_2$. Note that the second component of $L'$ is the knot $K_L\subset Y_L$.

Since $\Delta_L(s,t)=1$, the $\Z\oplus \Z$-cover of the exterior of $L$ has trivial first homology. The Alexander polynomial of $L'$ is also trivial because the exteriors of $L$ and $L'$ are homeomorphic, so their $\Z\oplus \Z$-covers are homeomorphic. Since the linking number of $L'$ is~1, $\Delta_{K_L}(t)=\Delta_{L'}(1,t)=1$ by Torres' condition \cite[page 57]{Torres:1953-1}. 
\end{proof}

\begin{lemma}\label{lemma:Zconcordance}Let $L=L_1\sqcup L_2$ be a $2$-component link with linking number $1$. If $L$ is smoothly concordant to a $2$-component link $J$ with trivial Alexander polynomial, then the knots $K_L$ and $K_J$ are smoothly concordant in a smooth homology $S^3\times [0,1]$ and $\Delta_{K_J}(t)=1$.
\end{lemma}
\begin{proof}Suppose that $L$ is  smoothly concordant to $J$ with $\Delta_J(s,t)=1$ via a smooth concordance $C_1\,\sqcup\, C_2\subset S^3\times [0,1]$. Let $Z$ be the result of $1$-surgery of $S^3\times [0,1]$ along $C_1$. Note that $\d Z=Y_L\sqcup -Y_J$. By Alexander duality and Mayer-Vietoris sequence, $Z$ is a smooth homology $S^3\times [0,1]$. The image of $C_2$ in $Z$ gives a smooth concordance between the results of surgery $K_L$ and $K_J$ in $Z$. By Lemma~\ref{lemma:Alexofblowdown}, $\Delta_{K_J}(t)=1$. 
\end{proof}
\begin{remark}An analogous statement of Lemma~\ref{lemma:Zconcordance} also holds in the topological category.\end{remark}

\subsection{Hedden-Livingston-Ruberman obstruction and its link analogue}
Theorem \ref{theorem:B} is inspired by the main obstruction theorem of \cite{Hedden-Livingston-Ruberman:2012-1} which we recall in Theorem~\ref{theorem:Hedden-Livingston-Ruberman}. For the reader's convenience, we recall some necessary definitions following \cite[Sections 2--3]{Hedden-Livingston-Ruberman:2012-1}. For more details, see \cite{Hedden-Livingston-Ruberman:2012-1}.

Let $Y$ be a $\Z_2$-homology 3-sphere. Recall that $Y$ has a non-singular $\Q/\Z$-valued linking form 
\[\lambda_Y\colon H_1(Y)\times H_1(Y)\to \Q/\Z,\]
which is the adjoint of the following composition of isomorphisms from Poincar\'{e} duality, the Bockstein long exact sequence and the universal coefficient theorem,
\[H_1(Y)\xrightarrow{\operatorname{PD}}H^2(Y)\xrightarrow{\beta^{-1}} H^1(Y;\Q/\Z)\xrightarrow{\operatorname{UCT}} \Hom_\Z(H_1(Y),\Q/\Z).\]
A subgroup $M$ of $H_1(Y)$ is called a \emph{metabolizer for the linking form} $\lambda_Y\colon H_1(Y)\to H_1(Y)\to \Q/\Z$ if $M=M^\perp$ where
\[M^\perp=\{x\in H_1(Y)\mid \lambda_Y(x,y)=0\textrm{ for all }y\in M\}.\]
\begin{definition}For $z\in H_1(Y)$, let $\mathfrak{s}_z$ be the unique Spin$^c$ structure of $Y$ which satisfies $c_1(\mathfrak{s}_z)=2\hat{z}\in H^2(Y)$ where $\hat{z}$ is the Poincar\'{e} dual of $z$. In particular, $\mathfrak{s}_0$ is the unique Spin structure on $Y$. Define $\overline{d}(Y,\mathfrak{s}_z):=d(Y,\mathfrak{s}_z)-d(Y,\mathfrak{s}_0)$.
\end{definition}
We remark that the $\overline{d}$-invariant that we use in this paper is different from the $\overline{d}$-invariant coming from involutive Heegaard Floer homology \cite{Hendricks-Manolescu:2017-1}.

We will use the following theorem which is analogous to \cite[Theorem~3.2]{Hedden-Livingston-Ruberman:2012-1}. The proof of \cite[Theorem~3.2]{Hedden-Livingston-Ruberman:2012-1} based on \cite[Proposition~2.1]{Hedden-Livingston-Ruberman:2012-1} can be easily adapted, so we leave the proof of Theorem~\ref{theorem:Hedden-Livingston-Ruberman} to the reader.
\begin{theorem}\label{theorem:Hedden-Livingston-Ruberman} Let $Y$ be a $\Z_2$-homology $3$-sphere. If there is a smooth $\Z_2$-homology $4$-ball $W$ such that $\d W=Y\# Z$ for some homology $3$-sphere $Z$, then there is a metabolizer $M\subset H_1(Y)$ for the linking form 
\[\lambda_Y\colon H_1(Y)\times H_1(Y)\to \Q/\Z\]
 such that $\overline{d}(Y,\mathfrak{s}_m)=0$ for all $m\in M$.
\end{theorem}
Now we prove Theorem \ref{theorem:B} whose precise version is given in Theorem \ref{theorem:mainobstruction}.
\begin{theorem}\label{theorem:mainobstruction}Let $L=L_1\sqcup L_2$ be a $2$-component link with linking number $1$. Let $Y_L$ be the $3$-manifold obtained by doing $1$-surgery on $S^3$ along $L_1$. Denote the image of $L_2$ in $Y_L$ by $K_L$. Let $\Sigma_{K_L}$ be the $2$-fold cover of $Y_L$ branched along $K_L$.  If $L$ is smoothly concordant to a link with trivial Alexander polynomial, then there is a metabolizer $M$ for the linking form 
\[\lambda_{\Sigma_{K_L}}\colon H_1(\Sigma_{K_L})\times H_1(\Sigma_{K_L})\to \Q/\Z\]
such that $\overline{d}(\Sigma_{K_L},\mathfrak{s}_m)=0$ for all $m\in M$.
\end{theorem}
\begin{proof}We continue to use notations used in the proof of Lemma \ref{lemma:Zconcordance}. Suppose that $L$ is concordant to a 2-component link $J$ with trivial Alexander polynomial via a concordance $C_1\sqcup C_2$. By doing 1-surgery on the first component $C_1$ of the concordance, we obtain a concordance $C$ from $K_L$ to $K_J$ in $Z$ where $Z$ is a $\Z$-homology $S^3\times [0,1]$. By Lemma \ref{lemma:Zconcordance}, the knot $K_J$ satisfies $\Delta_{K_J}(t)=1$. The double branched cover $\Sigma_C$ of $Z$ branched along $C$ is a $\Z_2$-homology cobordism from $\Sigma_{K_L}$ to $\Sigma_{K_J}$. Since $\Delta_{K_J}(t)=1$, $\Sigma_{K_J}$ is an integral homology 3-sphere. By removing an arc in $\Sigma_{C}$ joining $\Sigma_{K_L}$ and $\Sigma_{K_J}$, we obtain a smooth $\Z_2$-homology 4-ball $W$ whose boundary is $\Sigma_{K_J}\# -\Sigma_{K_L}$. By applying Theorem~\ref{theorem:Hedden-Livingston-Ruberman}, the conclusion follows.
\end{proof}

\section{Computation of $\overline{d}$-invariants}\label{section:d-invariants}
Our computation of $\overline{d}$-invariants has several ingredients.
\subsection{Knot Floer complexes}
Heegaard Floer homology associates a filtered chain complex $CF^\infty$ to an appropriate Heegaard diagram for a closed three-manifold \cite{Ozsvath-Szabo:2004-0}. We call this filtration the {\it algebraic filtration}, to distinguish it below. The homology of this and various sub- and quotient complexes are invariants of the three-manifold. From this, Ozsv\'ath and Szab\'o define a correction term \( d(Y,\mathfrak{t})\in \mathbb{Q} \) associated to a rational homology sphere $Y$ with Spin$^c$ structure $\mathfrak{t}$ \cite{Ozsvath-Szabo:2003-2}.

Ozsv\'ath and Szab\'o, and independently Rasmussen, showed that a knot $K\subset Y$ can be used to define a second filtration, which we call the {\it Alexander filtration} and denote by $A$, yielding a $\Z\oplus\Z$-filtered chain complex $CFK^\infty(Y,K)$, defined up to filtered chain homotopy equivalence \cite{Ozsvath-Szabo:2004-1,Rasmussen:2003-1}.  In the case of $K\subset S^3$, we will simply write $CFK^\infty(K)$. We will denote the grading on this complex by $M$. This complex can be used to compute Heegaard Floer homology of surgeries along $K$ \cite{Ozsvath-Szabo:2008-1,Ozsvath-Szabo:2011-1}. 

We will represent $CFK^\infty(K)$ in the $(i,j)$-plane. The complex is finitely generated over $\mathbb{F}[U,U^{-1}]$, where $U$ is a formal variable, and $\mathbb{F}$ denotes the field with two elements. Each generator $x$ has an Alexander filtration level $A(x)$, and we represent $x$ as a dot at $(0,A(x))$, and, for $i\in \Z$, we represent the homogeneous element $U^ix$ as a dot at $(-i,A(x)-i)$ (that is, multiplication by $U$ lowers each of the filtration levels by 1). The differential $\partial(U^ix)$ is a finite sum of homogeneous elements, which we represent by drawing an arrow from $U^ix$ to each element. Thus the Alexander filtration is seen vertically in the plane, and the algebraic filtration is seen horizontally; the filtration is manifested by the fact that the arrows must not go up or to the right. See Figure \ref{figure:cfktrefoil} for an example. 

\begin{figure}
\begin{tikzpicture}
\filldraw[black] (0,1) circle (2pt) node[anchor=south west]{$x$};
\filldraw[black] (-1,0) circle (2pt) node[anchor=south west]{$Ux$};
\filldraw[black] (0,0) circle (2pt) node[anchor=south west]{$y$};
\filldraw[black] (0,-1) circle (2pt) node[anchor=south west]{$z$};
\filldraw[black] (1,1) circle (2pt);
\filldraw[black] (1,0) circle (2pt);
\filldraw[black] (-1,-1) circle (2pt);
\filldraw[black] (-1,-2) circle (2pt);
\filldraw[black] (-2,-1) circle (2pt);
\node at (0,-3.5) {$i=0$};
\node at (2.5,3) {$\underline{A}$};
\node at (2.5,2) {$2$};
\node at (2.5,1) {$1$};
\node at (2.5,0) {$0$};
\node at (2.5,-1) {$-1$};
\node at (2.5,-2) {$-2$};
\draw[dotted,->] (1.2,1.2)--(1.8,1.8);
\draw[dotted,->] (-2,-2)--(-2.6,-2.6);
\draw[->] (1,1)--(1,0.2);
\draw[->] (1,1)--(0.2,1);
\draw[->] (0,0)--(0,-0.8);
\draw[->] (0,0)--(-0.8,0);
\draw[->] (-1,-1)--(-1,-1.8);
\draw[->] (-1,-1)--(-1.8,-1);
\end{tikzpicture}
\caption{ The knot Floer complex of the right-handed trefoil, $CFK^\infty(T_{2,3})$. It is freely generated over $\mathbb{F}[U,U^{-1}]$ by $x, y$ and $z$. The differential of both $x$ and $z$ vanishes, while $\partial y = Ux+z$. The homological gradings (not shown) are determined by the fact that $M(x)=0$, $\partial$ decreases $M$ by 1, and multiplication by $U$ decreases $M$ by 2. }
\label{figure:cfktrefoil}
\end{figure}
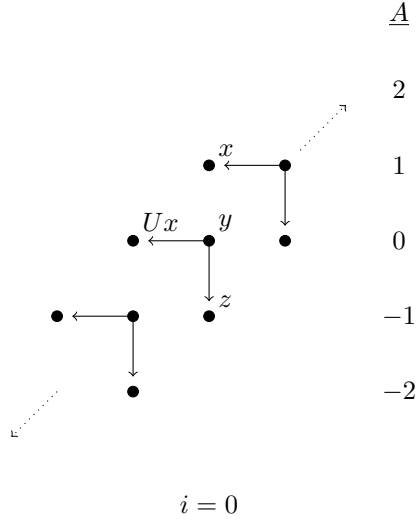

Suppose $S$ is a subset of $\Z\oplus \Z$ with the property that \[ (i,j)\in S \implies (i-m,j-n) \in S \text{ for all } m,n \geq 0.\] Then the subset of $C$ generated (over $\mathbb{F}$) by the elements with $(i,j)$-coordinates in $S$ is a filtered subcomplex, which we will denote $CS$. Thus $C\{i\leq 0\}$ is the subcomplex consisting of everything on or to the left of the $j$-axis, and $C\{i\leq 0, j\leq k\}$ is the ``third quadrant" shaped subcomplex, as seen in Figure \ref{figure:cfksubset}. 

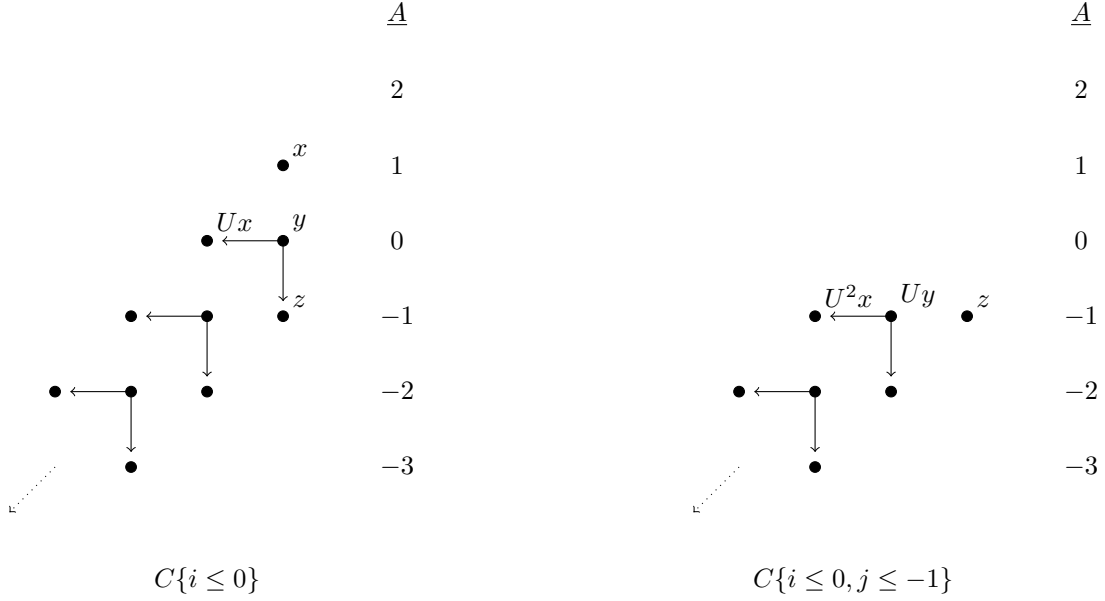
\begin{figure}
\begin{tikzpicture}
\filldraw[black] (0,1) circle (2pt) node[anchor=south west]{$x$};
\filldraw[black] (0,0) circle (2pt) node[anchor=south west]{$y$};
\filldraw[black] (0,-1) circle (2pt) node[anchor=south west]{$z$};
\filldraw[black] (-1,0) circle (2pt) node[anchor=south west]{$Ux$};
\filldraw[black] (9,-1) circle (2pt) node[anchor=south west]{$z$};
\filldraw[black] (8,-1) circle (2pt) node[anchor=south west]{$Uy$};
\filldraw[black] (7,-1) circle (2pt) node[anchor=south west]{$U^2x$};
\filldraw[black] (-1,-1) circle (2pt);
\filldraw[black] (-1,-2) circle (2pt);
\filldraw[black] (-2,-1) circle (2pt);
\filldraw[black] (-2,-2) circle (2pt);
\filldraw[black] (-2,-3) circle (2pt);
\filldraw[black] (-3,-2) circle (2pt);
\filldraw[black] (8,-2) circle (2pt);
\filldraw[black] (7,-2) circle (2pt);
\filldraw[black] (7,-3) circle (2pt);
\filldraw[black] (6,-2) circle (2pt);
\node at (-1,-4.5) {$C\{ i\leq 0\}$};
\node at (7.5,-4.5) {$C\{ i\leq 0,j\leq -1\}$};
\node at (1.5,3) {$\underline{A}$};
\node at (1.5,2) {$2$};
\node at (1.5,1) {$1$};
\node at (1.5,0) {$0$};
\node at (1.5,-1) {$-1$};
\node at (1.5,-2) {$-2$};
\node at (1.5,-3) {$-3$};
\node at (10.5,3) {$\underline{A}$};
\node at (10.5,2) {$2$};
\node at (10.5,1) {$1$};
\node at (10.5,0) {$0$};
\node at (10.5,-1) {$-1$};
\node at (10.5,-2) {$-2$};
\node at (10.5,-3) {$-3$};
\draw[->] (0,0)--(0,-0.8);
\draw[->] (0,0)--(-0.8,0);
\draw[->] (-1,-1)--(-1,-1.8);
\draw[->] (-1,-1)--(-1.8,-1);
\draw[->] (-2,-2)--(-2,-2.8);
\draw[->] (-2,-2)--(-2.8,-2);
\draw[->] (8,-1)--(8,-1.8);
\draw[->] (8,-1)--(7.2,-1);
\draw[->] (7,-2)--(7,-2.8);
\draw[->] (7,-2)--(6.2,-2);
\draw[dotted,->] (-3,-3)--(-3.6,-3.6);
\draw[dotted,->] (6,-3)--(5.4,-3.6);
\end{tikzpicture}
\caption{ Filtered subcomplexes of $C=CFK^\infty(T_{2,3})$.}
\label{figure:cfksubset}
\end{figure}

Our interest in knot Floer complexes here will be to compute $d$-invariants of surgeries. Ozsv\'ath and Szab\'o showed that the Heegaard Floer homology of $S^3_{p/q}(K)$ can be computed as the homology of a mapping cone involving such subcomplexes or their corresponding quotients. See \cite{Ozsvath-Szabo:2011-1} for full details;  we will review what is relevant for our purposes below.

\subsection{$V_i$-sequences}\label{section:visequence}

If $C = CFK^\infty(K)$ for some knot $K\subset S^3$, then \(H_*(C\{i\leq 0\})\) is isomorphic to a ``tower" $\mathcal{T^-}$;  that is, isomorphic to $\mathbb{F}[U]$, with 1 supported in grading zero.

Let \[v_k: C\{i\leq 0, j\leq k \} \to C\{ i\leq 0\} \] be inclusion. We define 
\begin{equation}\label{equation:vkdefinition}
V_k(K):= \max \{ M(x) \mid (v_k)_*(x) \in \mathcal{T}^- , \ (v_k)_*(x)\neq 0 \}, 
\end{equation}
where $M$ is the homological grading. Thus, from $CFK^\infty(K)$, we get a sequence of nonnegative integers $\{V_k(K)\}$. These were introduced by Rasmussen \cite[Definition 7.1 and following discussion]{Rasmussen:2003-1} (as ``local $h$-invariants"), and then, with the notation here, by Ni and Wu in their study of cosmetic surgeries \cite{Ni-Wu:2015-1}.

%These invariants were originally defined in terms of maps between {\it quotient} complexes of $CFK^\infty$ -- this equivalent definition in terms of subcomplexes, rather than quotient complexes, can be found in \cite[Section 3.2.2]{Hom:2017-1}, for example.

  Here we record some useful properties of the $V_k$'s.

\begin{proposition}[{\cite[Proposition 1.6 and Remark 2.10]{Ni-Wu:2015-1}}]\label{proposition:Viformula}Let $p,q$ be relatively prime integers and $i\in \Z_p$. Denote the unknot by $U$. For any knot $K$,  \[ d(S^3_{p/q}(K),\mathfrak{t}_i) = d(S^3_{p/q}(U),\mathfrak{t}_i)-2\max \{V_{\lfloor\frac{i}{q}\rfloor}(K),V_{\lfloor\frac{p+q-1-i}{q}\rfloor}(K)\}. \]
\end{proposition}
Here $\mathfrak{t}_i$ above refers to a specific Spin$^c$-structure; see Remark \ref{remark:spinclabel} for a discussion of this label.
\begin{remark}\label{remark:spinclabel} For any knot $K$, there is an identification $\varphi\colon \operatorname{Spin}^c(S^3_{p/q}(K))\to\Z_p$ given in \cite[Sections 4 and 7]{Ozsvath-Szabo:2011-1}. Denote the Spin$^c$ structure of $S^3_{p/q}(K)$ which corresponds to $i\in \Z_p$ under $\varphi$ by $\mathfrak{t}_i$. We recall some facts on $\mathfrak{t}_i$ (see \cite[Appendix~B]{Cochran-Horn:2015-1}).  The difference between two Spin$^c$ structures of $S^3_{p/q}(K)$ is given by the formula 
\[\mathfrak{t}_{j}=\mathfrak{t}_i+(j-i)q^*\hat{\mu}\]
where $q^*$ is an integer such that $qq^*\equiv 1\pmod{p}$ and $\hat{\mu}\in H^2(S^3_{p/q}(K))$ is the Poincar\'{e} dual of the homology class of a meridian of $K$. When $p$ is odd, it is known that the Spin structure of $S^3_{p/q}(K)$ is $\mathfrak{t}_i$ where $i$ is the mod $p$ reduction of the unique element in $\Z\cap \{\frac{q-1}{2},\frac{p+q-1}{2}\}$.
\end{remark}

\begin{proposition}[{\cite[Proposition 7.6]{Rasmussen:2003-1}}]\label{proposition:Videcreasing}
For each $k$, \( V_k(K) -1 \leq V_{k+1}(K) \leq V_k(K)\).
\end{proposition}

\begin{proposition}[{\cite[Proposition 6.1]{Bodnar-Celoria-Golla:2015-1}}]\label{proposition:additivity}
For any knots $K,J\subset S^3$ and any integers $k$ and $j$, \[ V_{k+j}(K\#J) \leq V_k(K)+V_j(J).\]
\end{proposition}
Note that the correction terms of lens spaces which appear in Proposition \ref{proposition:Viformula} are described by an inductive formula \cite[Proposition 4.8]{Ozsvath-Szabo:2003-2}, so the $V_k$'s determine the correction terms of rational surgeries. 
The following is a consequence of Proposition \ref{proposition:Viformula} and the fact that $d$ is a Spin$^c$ rational homology cobordism invariant.

\begin{proposition}\label{proposition:Viconcordance}
 For each $k$, $V_k(K)$ is a smooth concordance invariant of $K$.
\end{proposition}

\subsection{$\nu^+$-equivalence}
Following \cite{Hom-Wu:2016-1} and \cite{Kim-Park:2016-1}, we recall the $\nu^+$-invariant and $\nu^+$-equivalence.
\begin{definition}[\cite{Hom-Wu:2016-1}] For a knot $K$, $\nu^+(K)$ is the smallest integer $k$ such that $V_k(K)=0$. \end{definition}
\begin{proposition}[{\cite[Proposition 2.3]{Hom-Wu:2016-1}}]\label{proposition:whennu+iszero} The invariant $\nu^+$ is a smooth concordance invariant. For any knot $K$, $\nu^+(K)\geq 0$ and the equality holds if and only if $V_0(K)=0$.
\end{proposition}
\begin{definition}[{\cite{Kim-Park:2016-1}}] We say two knots $K_0$ and $K_1$ are \emph{$\nu^+$-equivalent} if 
\[\nu^+(K_0\#-K_1)=\nu^+(-K_0\# K_1)=0.\] 
\end{definition}
%Since the $V_k$'s and therefore $\nu^+$ are filtered chain homotopy invariants of the knot Floer complex, we will also say that $CFK^\infty(K_0)$ and $CFK^\infty(K_1)$ are $\nu^+$-equivalent.
\begin{remark}\label{remark:nu+equivalence}By Proposition \ref{proposition:whennu+iszero}, two knots $K_0$ and $K_1$ are $\nu^+$-equivalent if and only if \[V_0(K_0\# -K_1)=V_0(-K_0\# K_1)=0.\]
\end{remark}

\begin{remark}\label{remark:Hom}Hom's argument in the proof of \cite[Proposition 3.11]{Hom:2017-1} shows that two knots are $\nu^+$-equivalent if and only if $CFK^\infty(K_0)\oplus A_0$ and $CFK^\infty(K_1)\oplus A_1$  are filtered chain homotopy equivalent for some acyclic $A_0$ and $A_1$. (For details, see \cite[Lemma 3.1]{Kim-Park:2016-1}.)
\end{remark}

Though the following propositions seem to be well-known to experts, we prove them for the reader's convenience.
\begin{proposition}\label{proposition:nu+impliesVi}If $K$ and $J$ are $\nu^+$-equivalent, then $V_i(K)=V_i(J)$ for any $i\geq 0$.
\end{proposition}
\begin{proof}By Remark \ref{remark:nu+equivalence}, $V_0(K\# -J)=V_0(-K\#J)=0$. By Proposition \ref{proposition:additivity} and concordance invariance of $V_i$, we have 
\[V_i(K)=V_i(K\# -J\# J)\leq V_0(K\# -J)+V_i(J)=V_i(J).\]
(Note that the first equality uses the fact that $-J\# J$ is slice.) The same argument gives the opposite inequality $V_i(J)\leq V_i(K)$.
\end{proof}
\begin{proposition}\label{proposition:nu+andconnectedsum}Suppose that $K_i$ and $J_i$ are $\nu^+$-equivalent knots for $i=0,1$. Then  $K_0\# K_1$ and $J_0\# J_1$ are also $\nu^+$-equivalent. In particular, if $K$ and $J$ are $\nu^+$-equivalent, then $nK$ and $nJ$ are $\nu^+$-equivalent for any integer $n$.
\end{proposition}
\begin{proof}By Remark \ref{remark:nu+equivalence}, $V_0(K_i\#-J_i)=V_0(-K_i\#J_i)=0$ for $i=0,1$. By Proposition \ref{proposition:additivity}, we have 
\[V_0(K_0\#K_1\#-J_0\#-J_1)\leq V_0(K_0\#-J_0)+V_0(K_1\#-J_1)=0.\]
A similar argument shows that $V_0(-K_0\#-K_1\#J_0\#J_1)=0$. By Remark \ref{remark:nu+equivalence}, the two knots $K_0\# K_1$ and $J_0\#J_1$ are also $\nu^+$-equivalent.
\end{proof}

Now we give some known example of $\nu^+$-equivalences.
\begin{example}\label{example:nu+equivalent} Let $T_{p,q}$ be the $(p,q)$-torus knot and $D$ be the (positive) Whitehead double of the trefoil knot. Then, for any integer $n\geq 1$, both $nD$ and $nT_{2,3}$ are $\nu^+$-equivalent to $T_{2,2n+1}$ by Proposition~6.1 and Theorem~B.1 of \cite{Hedden-Kim-Livingston:2016-1} and Remark \ref{remark:Hom}. (Compare \cite[Example~3.2]{Kim-Park:2016-1}.)
\end{example}

It is shown in \cite{Kim-Park:2016-1} that $\nu^+$-equivalence is preserved under any satellite operation with non-zero winding number. In particular, $\nu^+$-equivalence is preserved under cabling operations.
\begin{theorem}[{\cite[Theorem B]{Kim-Park:2016-1}}]\label{theorem:KP} Suppose that $P\subset S^1\times D^2$ is a pattern with non-zero winding number. If two knots $K_0$ and $K_1$ are $\nu^+$-equivalent, then $P(K_0)$ and $P(K_1)$ are $\nu^+$-equivalent.
\end{theorem}

\begin{remark}\label{remark:nucomplex} Note that $V_k$ and therefore $\nu^+$ are invariants of the filtered chain homotopy type of $CFK^\infty(K)$. They can likewise be defined for any {\it abstract infinity complex} $C$, without knowing whether it is realized by an actual knot (see \cite[Definition 6.1]{Hedden-Watson:2014-1}). We say two such complexes $C$ and $C'$ are $\nu^+$-equivalent if \[\nu^+(C\otimes (C')^*) = \nu^+(C'\otimes C^*)=0,\] where $^*$ denotes the dual complex. We will abuse notation slightly further and say that a complex $C$ and a knot $K$ are $\nu^+$-equivalent if $C$ and $CFK^\infty(K)$ are.
\end{remark}

\subsection{Reduced knot Floer complexes}\label{section:reduced}

The proof of Theorem \ref{theorem:C} will involve computing the $V_k$'s for a connect sum of knots. We discuss here how that can be done effectively when the summands are ($\nu^+$-equivalent to) $L$-space knots or their mirrors, using the reduced knot Floer complex of the second author. We refer the reader to \cite{Krcatovich:2015-1} for the definition in general, and here review the special case of an $L$-space knot.

By \cite[Theorem 1.2]{Ozsvath-Szabo:2005-1} and \cite[Remark~6.6]{Hom:2014-1}, when $K$ is an $L$-space knot, we have \[ CFK^\infty(K) \cong \mathbb{F}[U,U^{-1}]\otimes S_{(a_1,\ldots,a_n)},\] where $S_{(a_1,\ldots,a_n)}$ is a {\it staircase complex},  generated by $x_1,\ldots, x_{2n+1}$. Here $x_{2i+1}$ has filtration level \[(a_1+\cdots + a_i, a_1+\cdots + a_{n-i}),\]  and $x_{2i}$ has filtration level \[(a_1+\cdots + a_i , a_1+\cdots + a_{n-i+1}),\] where each $a_i$ is a positive integer. The differential is given by \[ \partial(x_{2i+1}) = 0, \ \ \partial(x_{2i}) = x_{2i-1}+x_{2i+1}.\] See Figure \ref{staircaseex} for an example. 
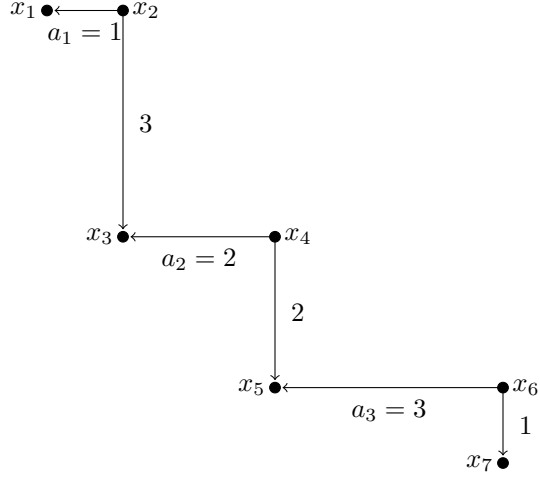
\begin{figure}
\begin{tikzpicture}
\filldraw[black] (0,6) circle (2pt) node[anchor=east]{$x_1$};
\filldraw[black] (1,6) circle (2pt) node[anchor=west]{$x_2$};
\filldraw[black] (1,3) circle (2pt) node[anchor=east]{$x_3$};
\filldraw[black] (3,3) circle (2pt) node[anchor=west]{$x_4$};
\filldraw[black] (3,1) circle (2pt) node[anchor=east]{$x_5$};
\filldraw[black] (6,1) circle (2pt) node[anchor=west]{$x_6$};
\filldraw[black] (6,0) circle (2pt) node[anchor=east]{$x_7$};
\draw[->] (1,6)--(0.1,6);
\draw[->] (1,6)--(1,3.1);
\draw[->] (3,3)--(1.1,3);
\draw[->] (3,3)--(3,1.1);
\draw[->] (6,1)--(3.1,1);
\draw[->] (6,1)--(6,0.1);
\node at (0.5,5.7) {$a_1=1$};
\node at (2,2.7) {$a_2=2$};
\node at (4.5,0.7) {$a_3=3$};
\node at (1.3,4.5) {$3$};
\node at (3.3, 2) {$2$};
\node at (6.3,0.5) {$1$};
\end{tikzpicture}
\caption{The staircase $S_{(1,2,3)}$, which generates $CFK^\infty(T_{4,5})$ over $\mathbb{F}[U,U^{-1}]$. The generator $x_1$ is in filtration level $(0,6)$. The $a_i$'s are the lengths (from left to right) of the horizontal arrows, and also (from bottom to top) of the vertical arrows.}\label{staircaseex}
\end{figure}
In the case that this complex belongs to a knot $K$,  $g(K)=\sum_i a_i$ is the Seifert genus of the knot. Note that all the generators with odd indices are pairwise homologous, and any of them generates the homology of the staircase complex. The list $(a_1,\ldots, a_n)$ is related to the Alexander polynomial of $K$ as follows. All nonzero terms in $\Delta_K(t)$ have coefficients $\pm 1$, with signs alternating, and the highest degree term positive. The number $a_i$ is the difference between the $(2i-1)$th and the $2i$th exponent. Because $\Delta_K(t)$ is symmetric, this list determines $\Delta_K(t)$, so we will say the list corresponds to $\Delta_K(t)$, or further that the list {\it corresponds to $K$}.

Given an $L$-space knot $K$, let $C = CFK^\infty(K)\{i\leq 0 \}.$  Then the reduced complex $\underline{CFK}^-(K)$ consists merely of the generators in $C$ which have no outgoing nor incoming horizontal arrows, and has trivial differential (see \cite[Corollary 4.2]{Krcatovich:2015-1}). As a complex, $\underline{CFK}^-(K)\cong \mathbb{F}[U]$, with 1 supported in grading zero. The Alexander filtration descends to a filtration on the reduced complex. The complex no longer has an algebraic filtration, but still has the structure of an $\mathbb{F}[U]$-module, with multiplication by $U$ taking the generator in grading $2i$ to the generator in grading $2i-2$. 

\begin{remark} \label{remark:reducedfiltration}The (Alexander) filtration on the reduced complex of an $L$-space knot can be determined explicitly from the list $(a_1,a_2,\ldots,a_n)$. If $x$ is the generator (over $\mathbb{F}[U]$) of $\underline{CFK}^-(K)$, then $x$ has filtration level $g(K)=a_1+\cdots +a_n$, and if \[ a_1+\cdots +a_i \leq j < a_1+\cdots + a_{i+1},\] then $U^jx$ has filtration level \[ a_1+\cdots + a_{n-i} -j .\] For $j\geq g(K), U^jx$ has filtration level $-j$. 
\end{remark}

\begin{figure}
$$\xymatrixcolsep{0.65 pc}\xymatrixrowsep{0.65 pc}
\xymatrix{ & & & & & & & & \underline{A} \\
& & & & & & & \tc{red}{\bullet} & 6 \\
& & & & & & \bullet & \bullet \ar[l] \ar[ddd] & 5 \\
& & & & & \bullet & \bullet \ar[l] \ar[ddd] & & 4\\
& & & & \bullet & \bullet \ar[l] \ar[ddd] & & & 3 \\
& & & \bullet & \bullet \ar[l] \ar[ddd] & & & \tc{red}{\bullet} & 2 \\
& & \bullet & \bullet \ar[l] \ar[ddd] & & & \tc{red}{\bullet}  &  & 1 \\
& \bullet & \bullet \ar[l]  \ar[ddd] & & & \bullet & & \bullet \ar[ll] \ar[dd] & 0\\
\bullet & \bullet \ar[l]  \ar[ddd] & & & \bullet & & \bullet \ar[ll] \ar[dd] & & -1\\
& & & \bullet & & \bullet \ar[ll] \ar[dd] & & \tc{red}{\bullet} & -2 \\
& & \bullet & & \bullet \ar[ll] \ar[dd] & & \tc{red}{\bullet} & & -3 \\
& \bullet & & \bullet \ar[ll] \ar[dd] & & \tc{red}{\bullet} & & & -4 \\
& & \dar[ddll] & & \bullet & & & \bullet \ar[lll] \ar[d] & -5 \\
& & &  \bullet & & & \bullet \ar[lll] \ar[d] & \tc{red}{\bullet} & -6 \\
& & & & & & \tc{red}{\bullet} & & -7
} \hspace{3cm} \xymatrix{ & \underline{A}\\
\bullet & 6 \\
& 5 \\
& 4 \\
& 3 \\
\bullet & 2 \\
\bullet & 1 \\
& 0 \\
& -1 \\
\bullet & -2 \\
\bullet & -3 \\
\bullet & -4 \\
& -5\\
\bullet & -6 \\
\bullet \dar[d] & -7 \\
&
}$$
\caption{At the left is $CFK^\infty(T_{4,5})\{i\leq 0 \}$, and at the right is the corresponding reduced complex, which consists of the elements colored red on the left; those with no incoming or outgoing horizontal arrows. Multiplication by $U$ in the reduced complex takes one dot downward to the next.}\label{lspacereduced}
\end{figure}
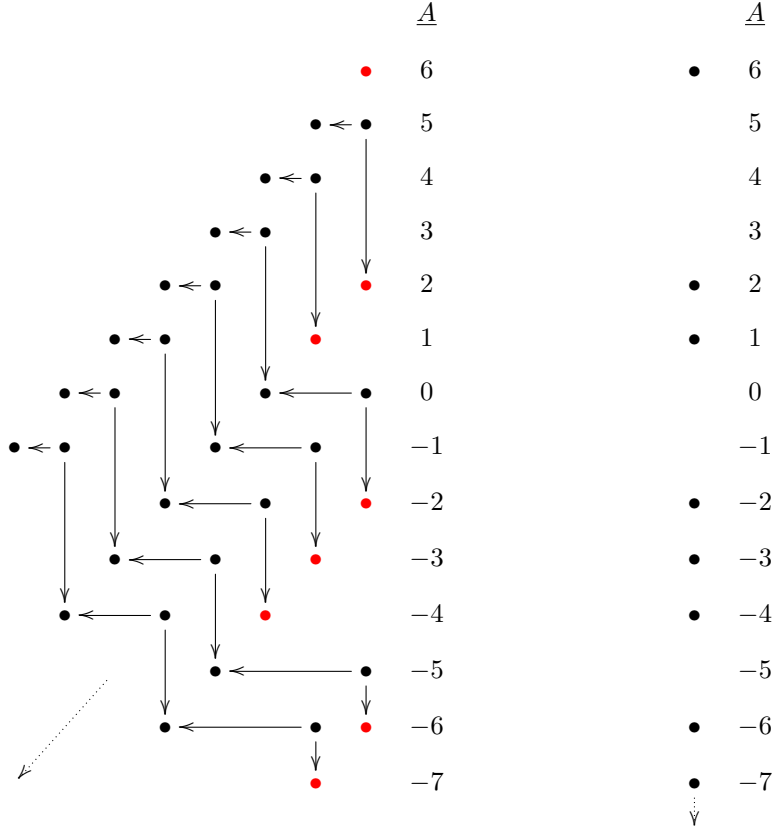

We can compute the $V_k$'s from the reduced complex $\underline{C}$ as 
\begin{equation}\label{equation:Vifromreduced}
 V_k(K) = -\tfrac{1}{2} \max \{ \text{grading of a non } U\text{-torsion generator of } H_*(\underline{C}\{j\leq k \})\} .
\end{equation}

For $K$ an $L$-space knot, $\underline{C}$ is isomorphic to its homology, but Equation \eqref{equation:Vifromreduced} holds for all knots, where this is not in general the case.

\begin{proposition}[{\cite[Theorem 3.4]{Krcatovich:2015-1}}]\label{proposition:reducedsum}
If $C$ is a complex with reduction $\underline{C}$, then there is a $(\mathbb{Z},U)$-filtered chain homotopy equivalence between $C\otimes CFK^-(K_2)$ and $\underline{C}\otimes CFK^-(K_2)$ for any knot $K_2$. In particular, if $C$ is $\nu^+$-equivalent to $CFK^\infty(K_1)$ for some knot $K_1$, then $V_k(K_1\# K_2)$ can be computed from $\underline{C}\otimes CFK^-(K_2)$ using \eqref{equation:Vifromreduced}.
\end{proposition}

\subsection{Staircases and connect sums}

The staircases corresponding to two $L$-space knots can be combined to produce a ``representative staircase", which can then be used to compute the $V_k$'s of the connect sum of the two knots. This idea is introduced in \cite[Section~5.1]{Borodzik-Livingston:2014-1}, see also \cite[Example 2]{Krcatovich:2015-1}, \cite[Section~7]{Golla-Marengon:2016-1}. Here we prove the following statement, which is stronger where it applies, as it will in our case of interest.

%Suppose $K$ is an $L$-space knot with staircase $S_{(a_1,a_2,\ldots,a_n)}$ generated by $x_1,x_2,\ldots,x_{2n+1}$, and $J$ an $L$-space knot with staircase $S_{(b_1,b_2,\ldots,b_m)}$ generated by $y_1,y_2,\ldots, y_{2m+1}$. By \textcolor{purple}{cite OS}, \[C:=CFK^\infty(K\# J) \cong CFK^\infty(K) \otimes CFK^\infty (J),\] so $C$ has generators of the form $x_i\otimes y_j$. Note that any of the $x_{2i+1}\otimes y_{2j+1}$ generates homology of $C$. Now let $Z$ be the set of coordinates $(i,j)$ for which
%\begin{itemize}
%\item $C$ has a generator $x_{2k+1}\otimes y_{2j+1}$ in filtration level $(i,j)$
%\item if $C$ has a generator in filtration level $(i',j')$, then either $i< i'$ or $j< j'$ $\left( \text{ or } (i',j')=(i,j)\right)$.
%\end{itemize}
%Then there exists a (unique) staircase $S^{\#}$ which has odd generators whose filtration levels are exactly the elements of $Z$. We call this staircase the representative staircase for $K\# J$. 
%\begin{proposition}\label{proposition:staircasesum}
%If $K$ and $J$ are $L$-space knots, and $K\# J$ has representative staircase $S^{\#}$, then $CFK^\infty(K\# J)$ is $\nu^+$-equivalent to $\mathbb{F}[U,U^{-1}] \otimes S^{\#}.$
%\end{proposition}

\begin{lemma}\label{lemma:staircasesummand}
If $K$ and $J$ are $L$-space knots whose staircase shapes admit a compatible riffle, then \[CFK^\infty(K\# J) \cong \left(\mathbb{F}[U,U^{-1}]\otimes S^{\#}\right) \oplus A,\] where $S^{\#}$ is the staircase complex corresponding to the riffled shape, and $A$ is an acyclic subcomplex. In particular, $CFK^\infty(K\# J)$ is $\nu^+$-equivalent to a staircase complex.
\end{lemma}
The $S^{\#}$ above will be called the {\it representative staircase} for the sum. Before proving, we introduce the necessary terminology.

A {\it riffle} of two ordered lists $A$ and $B$ is an ordering of $A\cup B$ which restricts to the given orderings on each sublist $A$ and $B$. In a riffle of $A$ and $B$, the {\it opposite successor} of an element $a\in A$ (respectively $b\in B$) is the first element of $B$ (resp. $A$) which appears after $a$ (resp. $b$), if such an element exists. As an example, \[ ( 1,2,3,a,b,4,c,5,d,e ) \] is a riffle of $( 1,2,3,4,5)$ and $( a,b,c,d,e)$; the opposite successor of $1$ is $a$, the opposite successor of $a$ is $4$, and $d$ and $e$ have no opposite successor.

Given a staircase with list $(a_1,\ldots, a_n)$, we define the {\it staircase shape} as the ordered list of ordered pairs \[\left( (a_1,a_n),(a_2,a_{n-1}),\ldots, (a_n,a_1)\right). \]

Two staircase shapes are {\it compatible} if thye can be riffled so that if $(b_j,b_{m-j+1})$ is the opposite successor of $(a_i,a_{n-i+1})$, then $a_i\leq b_j$ and $a_{n-i+1}\geq b_{m-j+1}$ (and likewise, if $(a_i,a_{n-i+1})$ is the opposite successor of $(b_j,b_{m-j+1})$, then  $a_i\geq b_j$ and $a_{n-i+1}\leq b_{m-j+1}$). We will also call such a riffle compatible.

\begin{proof}[Proof of Lemma \ref{lemma:staircasesummand}]

Let us assume that $CFK^\infty(K)\cong \mathbb{F}[U,U^{-1}]\otimes S_{(a_1,\ldots,a_n)}$, with generators $x_1$ through $x_{2n+1}$, and that $CFK^\infty(J)\cong \mathbb{F}[U,U^{-1}]\otimes S_{(b_1,\ldots,b_m)}$, with generators $y_1$ through $y_{2m+1}$. Thus $CFK^\infty(K\# J) \cong CFK^\infty(K) \otimes CFK^\infty (J)$ has basis $x_iy_j$ (for readability, we write these elements without the `$\otimes$' symbol).

We will choose a new basis for $CFK^\infty(K\# J)$ which exhibits the direct sum splitting. The compatibility condition will be precisely what we need to ensure that our change of basis is filtered. Since the staircases are compatible, we can find a riffle of the staircase shapes % \[ \left( (\delta_{i_1},\delta_{j_1}),\ldots , (\delta_{i_{n+m}},\delta_{j_{n+m}}) \right) \] 
which is compatible; fix one such riffle. The first generator of the representative staircase will be $x_1y_1$. Suppose the $(2k-1)$th generator on the staircase is $x_{2i+1}y_{2k-2i-1}$, for some $i$. Then, the $k$th element in the compatible riffle is either 
\begin{enumerate}[label=(\roman*)]
\item $(a_{i+1},a_{n-i})$ or
\item $(b_{k-i},b_{m-(k-i)+1})$.
\end{enumerate}

In case (i), the next step on the staircase consists of the generators $x_{2i+2}y_{2k-2i-1}$ and $x_{2i+3}y_{2k-2i-1}$. Further, we make the change of basis 
\begin{align*}
x_{2i+1}y_{2k-2i} \mapsto & \ x_{2i+2}y_{2k-2i-1}+x_{2i+1}y_{2k-2i},\\
x_{2i+1}y_{2k-2i+1} \mapsto &\  x_{2i+3}y_{2k-2i-1} + x_{2i+1}y_{2k-2i+1}.
\end{align*}

In case (ii), the next step in the staircase consists of $x_{2i+1}y_{2k-2i}$ and $x_{2i+1}y_{2k-2i+1}$. We make the filtered change of basis
\begin{align*}
x_{2i+2}y_{2k-2i-1} \mapsto & \ x_{2i+2}y_{2k-2i-1}+x_{2i+1}y_{2k-2i},\\
x_{2i+3}y_{2k-2i-1} \mapsto &\  x_{2i+3}y_{2k-2i-1} + x_{2i+1}y_{2k-2i+1}.
\end{align*}
This process terminates when the final step of the staircase includes the generator $x_{2n+1}y_{2m+1}$. 

We now have a collection of acyclic subcomplexes \[ A_{i,j}:= \{ x_{2i}y_{2j}, x_{2i-1}y_{2j} + x_{2i}y_{2j-1}, x_{2i+1}y_{2j} + x_{2i}y_{2j+1}, x_{2i-1}y_{2j+1} + x_{2i+1}y_{2j-1} \} , \] for $1\leq i \leq n$, $1\leq j \leq m$.

\begin{figure}
\begin{tikzpicture}
\filldraw[black] (0,1) circle (2pt) node[anchor=south]{$x_1$};
\filldraw[black] (1,1) circle (2pt) node[anchor=south]{$x_2$};
\filldraw[black] (1,0) circle (2pt) node[anchor=west]{$x_3$};
\draw[->] (1,1)--(0.1,1);
\draw[->] (1,1)--(1,0.1);
\filldraw[white] (2,1) circle (1pt) node[black]{$\otimes$};
\filldraw[white] (7,1) circle (1pt) node[black]{$\cong$};
\filldraw[black] (3,3) circle (2pt) node[anchor=south]{$y_1$};
\filldraw[black] (4,3) circle (2pt) node[anchor=south]{$y_2$};
\filldraw[black] (4,1) circle (2pt) node[anchor=east]{$y_3$};
\filldraw[black] (6,1) circle (2pt) node[anchor=south]{$y_4$};
\filldraw[black] (6,0) circle (2pt) node[anchor=west]{$y_5$};
\draw[->] (4,3)--(3.1,3);
\draw[->] (4,3)--(4,1.1);
\draw[->] (6,1)--(4.1,1);
\draw[->] (6,1)--(6,0.1);
\filldraw[red] (8,4) circle (2pt) node[anchor=east]{$x_1y_1$};
\filldraw[red] (9,4) circle (2pt) node[anchor=south east]{$x_1y_2$};
\filldraw[black] (10,4) circle (2pt) node[anchor=west]{$x_2y_2$};
\filldraw[black] (9.2,4) circle (2pt) node[anchor=north west]{$x_2y_1$};
\filldraw[black] (9,3) circle (2pt) node[anchor=north west]{$x_3y_1$};
\filldraw[black] (10,3) circle (2pt) node[anchor=west]{$x_3y_2$};
\filldraw[red] (9,2) circle (2pt) node[anchor=east]{$x_1y_3$};
\filldraw[red] (10,2) circle (2pt) node[anchor=north east]{$x_2y_3$};
\filldraw[black] (11,2) circle (2pt) node[anchor=south]{$x_1y_4$};
\filldraw[black] (12,2) circle (2pt) node[anchor=west]{$x_2y_4$};
\filldraw[red] (10,1) circle (2pt) node[anchor=east]{$x_3y_3$};
\filldraw[black] (11,1) circle (2pt) node[anchor=north]{$x_1y_5$};
\filldraw[black] (12,1.2) circle (2pt) node[anchor=south east]{$x_2y_5$};
\filldraw[red] (12,1) circle (2pt) node[anchor=north west]{$x_3y_4$};
\filldraw[red] (12,0) circle (2pt) node[anchor=west]{$x_3y_5$};
\draw[red, ->] (9,4)--(8.1,4);
\draw[red, ->] (9,4) to [out=260,in=100] (9,2.1);
\draw[red, ->] (10,2)--(9.1,2);
\draw[red,->] (10,2)--(10,1.1);
\draw[red, ->] (12,1)--(12,0.1);
\draw[red, ->] (12,1) to [out=190,in=350] (10.1,1);
\draw[->] (9.2,4)--(9,3.1);
\draw[->] (9.2,4) to [out=210, in=330] (8.1,4);
\draw[->] (10,3)--(9.1,3);
\draw[->] (10,3) to [out=290,in=70] (10,1.1);
\draw[->] (11,2)--(11,1.1);
\draw[->] (11,2) to [out=160, in=20] (9.1,2);
\draw[->] (12,1.2)--(11.1,1);
\draw[->] (12,1.2) to [out=240,in=120] (12,0.1);
\draw[->] (10,4)--(9.3,4);
\draw[->] (10,4)--(10,3.1);
\draw[->] (10,4) to [out=130,in=50] (9,4);
\draw[->] (10,4) to [out=330,in=30] (10,2);
\draw[->] (12,2)--(12,1.3);
\draw[->] (12,2)--(11.1,2);
\draw[->] (12,2) to [out=130,in=50] (10,2);
\draw[->] (12,2) to [out=320,in=40] (12,1);
\end{tikzpicture}\\
\caption{$T_{2,3}$ has staircase shape $((1,1))$, and $T_{3,4}$ has staircase shape $((1,2),(2,1))$. Generating sets for $CFK^\infty$ of each are shown on the left, and a generating set for the tensor product complex $CFK^\infty(T_{2,3})\otimes CFK^\infty(T_{3,4})$ is shown on the right. Note that the staircase shapes are compatible, with compatible riffle $((1,2),(1,1),(2,1))$, and this is the staircase shape of the subcomplex colored red.}\label{figure:repstaircase1}
\begin{tikzpicture}
\filldraw[black] (0, 8) circle (2pt) node[anchor=south]{$x_1y_1$};
\filldraw[black] (2, 8) circle (2pt) node[anchor=south east]{$x_1y_2$};
\filldraw[black] (2.3, 8) circle (2pt) node[anchor= north west]{$x_2y_1$};
\filldraw[black] (4, 8) circle (2pt) node[anchor=west]{$x_2y_2$};
\filldraw[black] (2.3, 6) circle (2pt);% node[anchor=north west]{$x_3y_1+x_1y_3$};
\filldraw[white] (2, 6) circle (2pt) node[black, anchor=north west]{$x_3y_1+x_1y_3$};
\filldraw[white] (2.3,7.5) circle (2pt) node[black, anchor= north west]{$+x_1y_2$};
\filldraw[black] (4, 6) circle (2pt) node[anchor=south west]{$x_2y_3+x_3y_2$};
\filldraw[black] (2, 4) circle (2pt) node[anchor=east]{$x_1y_3$};
\filldraw[black] (4, 4) circle (2pt) node[anchor=west]{$x_2y_3$};
\filldraw[black] (4, 2) circle (2pt) node[anchor=east]{$x_3y_3$};
\filldraw[black] (6, 2.5) circle (2pt) node[anchor=north]{$x_1y_5+x_3y_3$};
\filldraw[black] (6, 4) circle (2pt) node[anchor=south]{$x_1y_4+x_2y_3$};
\filldraw[black] (8, 2) circle (2pt) node[anchor=west]{$x_3y_4$};
\filldraw[black] (8, 2.5) circle (2pt) node[anchor=west]{$x_3y_4+x_2y_5$};
\filldraw[black] (8, 0) circle (2pt) node[anchor=west]{$x_3y_5$};
\filldraw[black] (8, 4) circle (2pt) node[anchor=west]{$x_2y_4$};
\draw[->] (2,8)--(0.1,8);
\draw[->] (2,8)--(2,4.1);
\draw[->] (4,4)--(2.1,4);
\draw[->] (4,4)--(4,2.1);
\draw[->] (8,2)--(4.1,2);
\draw[->] (8,2)--(8,0.1);
\draw[->] (8,4)--(8,2.6);
\draw[->] (8,4)--(6.1,4);
\draw[->] (8,2.5)--(6.1,2.5);
\draw[->] (6,4)--(6,2.6);
\draw[->] (4,8)--(2.4,8);
\draw[->] (4,8)--(4,6.1);
\draw[->] (2.3,8)--(2.3,6.1);
\draw[->] (4,6)--(2.4,6);
\end{tikzpicture}
\caption{After a filtered change of basis, $CFK^\infty(T_{2,3}\# T_{3,4})$ is seen to be generated by a direct sum of the staircase $S^{\#}$ (the subcomplex which was colored red in Figure~\ref{figure:repstaircase1}) and acyclic complexes.} \label{figure:repstaircase2}
\end{figure}

We have also constructed a subcomplex which is a staircase $S^{\#}$ whose corresponding list is a riffle of $(a_1,\ldots, a_n)$ and $(b_1, \ldots, b_m)$. Now we have \[CFK^\infty(K\# J) \cong \mathbb{F}[U,U^{-1}] \otimes \left( S^{\#} \oplus \bigoplus_{i,j} A_{i,j} \right),\]
which has the desired form.  Figures \ref{figure:repstaircase1} and \ref{figure:repstaircase2} provide an illustration of this construction. 

Note that such a change of basis can be performed on any tensor product of staircase complexes, but it may not always be {\it filtered} -- the compatibility assumption ensures that this is indeed a filtered change of basis, and the result can still be used to compute $V_k$'s. The final statement then follows from \cite[Proposition 3.11]{Hom:2017-1}.
\end{proof}

%As a consequence, if a compatible riffle exists, it is unique, up to swapping $(a_i, a_{n-i+1})$ and $(b_j, b_{m-j+1})$, in the %case where both $a_i=b_j$ and $a_{n-i+1}=b_{m-j+1}$. Suppose $S^{\#}$ and $S^{\#'}$ are each representative %staircases

\subsection{Proof of Theorem \ref{theorem:C}}
Now we prove Theorem \ref{theorem:C}.

\begin{lemma}\label{lemma:Jn}Let $D$ be the Whitehead double of the right handed trefoil knot. For any integer $n\geq 1$, let $J_n=(nD)_{2,4n-1} \# -T_{2,4n-1} \#2(n-1)D$. Then, the knot $J_n$ satisfies the following properties.
\begin{enumerate}[label={\upshape(\arabic*)}]
\item\label{item:Kntopslice} $J_n$ is topologically slice. 
\item\label{item:Knnu+} $2J_n$ is $\nu^+$-equivalent to $2T_{2,2n+1;2,4n-1}\# -T_{2,5}$.
\item\label{item:ViKn} $V_i(2J_n)=V_i(2T_{2,2n+1;2,4n-1}\# -T_{2,5})$ for any $i$.
\end{enumerate}
\end{lemma}
\begin{proof}\ref{item:Kntopslice} Since $nD$ has trivial Alexander polynomial, it is topologically slice by Freedman \cite{Freedman:1982-1}. For any topologically slice knot $J$, $J_{p,q}\# -T_{p,q}$ is topologically slice. In particular, $(nD)_{2,4n-1}\#-T_{2,4n-1}$ is topologically slice for any $n$. As a connected sum of two topologically slice knots, $J_n$ is topologically slice for any $n$.

\ref{item:Knnu+} By Example \ref{example:nu+equivalent}, for any $n\geq 1$, $nD$, $-T_{2,4n-1}$ and $2(n-1)D$ are $\nu^+$-equivalent to $T_{2,2n+1}$, $-(2n-1)T_{2,3}$ and $2(n-1)T_{2,3}$, respectively. Since $nD$ is $\nu^+$-equivalent to $T_{2,2n+1}$, $(nD)_{2,4n-1}$ is $\nu^+$-equivalent to $T_{2,2n+1;2,4n-1}$ by Theorem \ref{theorem:KP}. By Proposition \ref{proposition:nu+andconnectedsum}, $J_n$ is $\nu^+$-equivalent to $T_{2,2n+1;2,4n-1}\# -T_{2,3}$, and hence $2J_n$ is $\nu^+$-equivalent to $2T_{2,2n+1;2,4n-1}\# -T_{2,5}$ by Proposition \ref{proposition:nu+andconnectedsum} and Example \ref{example:nu+equivalent}.

\ref{item:ViKn} This follows from the item \ref{item:Knnu+} by Proposition \ref{proposition:nu+impliesVi}.
\end{proof}

By Lemma \ref{lemma:Jn}, we know $J_n$ is topologically slice for every $n$. To prove Theorem \ref{theorem:C}, it suffices to prove Theorem \ref{theorem:JnV0V1}.
\begin{theorem}\label{theorem:JnV0V1} For any $n\geq 1$, $V_0(2J_n)=2n$ and $V_1(2J_n)=2n-1$.
\end{theorem}
\begin{proof}It suffices to prove that $V_0(2K_n\# -T_{2,5})=2n$ and $V_1(2K_n\# -T_{2,5})=2n-1$ by Lemma~\ref{lemma:Jn}\ref{item:ViKn}, 
where $K_n = T_{2,2n+1;2,4n-1}$. By \cite[Theorem~1.10]{Hedden:2009-1}, $K_n$ is an $L$-space knot. We have that 
\begin{align*}
\Delta_{K_n}(t)=& \Delta_{T_{2,2n+1}}(t^2)\cdot \Delta_{T_{2,4n-1}}(t)\\
 =& \left( \sum_{i=0}^{2n} (-1)^i(t^2)^{i} \right) \left( \sum_{j=0}^{4n-2} (-1)^jt^j \right) \\
=& \sum_{i=0}^{n-1} \left( t^{4i}-t^{4i+2}\right)\left(\sum_{j=0}^{4n-2} (-1)^jt^j \right) + t^{4n}\left(\sum_{j=0}^{4n-2} (-1)^jt^j \right) \\
=& \sum_{i=0}^{n-1} \left(t^{4i}-t^{4i+1} + t^{4n+4i-1}-t^{4n+4i}\right) + \sum_{j=0}^{4n-2} (-1)^jt^{j+4n}.
\end{align*}
Note for convenience that, because Alexander polynomials are symmetric, this polynomial is  determined by its coefficients up to degree \( \frac{1}{2}(8n-2) = 4n-1\). Therefore, if we let \[ p_n(t)  = \sum_{i=0}^{n-1} t^{4i}-t^{4i+1},\] then we see that \[ \Delta_{K_n}(t) = p_n(t) + t^{4n-1}+t^{8n-2}p_n(t^{-1}).\]

For example, when $n=3$, we have \[p_3(t) = 1-t+t^4-t^5+t^8-t^9,\] so \[ \Delta_{K_3}(t) = 1-t+t^4-t^5+t^8-t^9 +t^{11} -t^{13}+t^{14}-t^{17}+t^{18}-t^{21}+t^{22}. \]

The polynomial $\Delta_{K_n}(t)$ defines a staircase complex with corresponding list \[ ( \overbrace{1,\ldots, 1}^{n},2,\underbrace{3,\ldots,3}_{n-1} ).\]
Thus, the staircase shape is \[ \left(   \overbrace{ (1,3),\ldots, (1,3)}^{n-1}, (1,2), (2,1),\underbrace{(3,1),\ldots,(3,1)}_{n-1} \right).\]
It is easily seen that a compatible riffle of the shapes for two copies of $K_n$ is given by \[ \left(    \overbrace{ (1,3),\ldots, (1,3)}^{2n-2}, (1,2),(1,2),(2,1), (2,1),\underbrace{(3,1),\ldots,(3,1)}_{2n-2} \right).\]
So, by Lemma \ref{lemma:staircasesummand}, the representative staircase for $K_n \# K_n$ has list \[ ( \overbrace{1,\ldots, 1}^{2n},2,2,\underbrace{3,\ldots,3}_{2n-2} ).\]
Up to $\nu^+$-equivalence then, we can replace $CFK^\infty(2K_n)$ with the staircase complex corresponding to this list, which we will denote by $C_n$.

The torus knot $T_{2,5}$ is also an $L$-space knot, which corresponds to the list $(1,1)$. For brevity, we let $C^-$ denote the complex for the mirror image, $CFK^-(-T_{2,5})$. A basis for $C^-$ is given in Figure~\ref{figure:cfkt25}. 
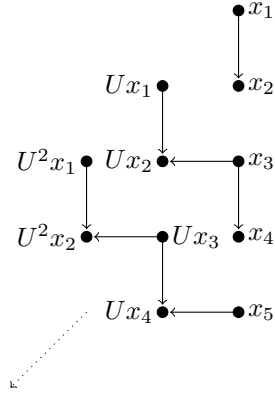
\begin{figure}
\begin{tikzpicture}
\filldraw[black] (0,2) circle (2pt) node[anchor=east]{$U^2x_1$};
\filldraw[black] (0,1) circle (2pt) node[anchor=east]{$U^2x_2$};
\filldraw[black] (1,1) circle (2pt) node[anchor=west]{$Ux_3$};
\filldraw[black] (1,0) circle (2pt) node[anchor=east]{$Ux_4$};
\filldraw[black] (1,3) circle (2pt) node[anchor=east]{$Ux_1$};
\filldraw[black] (1,2) circle (2pt) node[anchor=east]{$Ux_2$};
\filldraw[black] (2,2) circle (2pt) node[anchor=west]{$x_3$};
\filldraw[black] (2,1) circle (2pt) node[anchor=west]{$x_4$};
\filldraw[black] (2,4) circle (2pt) node[anchor=west]{$x_1$};
\filldraw[black] (2,3) circle (2pt) node[anchor=west]{$x_2$};
\filldraw[black] (2,0) circle (2pt) node[anchor=west]{$x_5$};
\draw[->] (0,2)--(0,1.1);
\draw[->] (1,1)--(0.1,1);
\draw[->] (1,1)--(1,0.1);
\draw[->] (2,0)--(1.1,0);
\draw[->] (1,3)--(1,2.1);
\draw[->] (2,2)--(1.1,2);
\draw[->] (2,2)--(2,1.1);
\draw[->] (2,4)--(2,3.1);
\draw[dotted,->] (0,0)--(-1,-1);
\end{tikzpicture}
\caption{$CFK^-(-T_{2,5})$.}\label{figure:cfkt25}
\end{figure}

\begin{figure}
\begin{tikzpicture}[scale=0.3]
\filldraw[black] (-7,14) circle (4pt);
\filldraw[black] (-6,14) circle (4pt);
\filldraw[black] (-6,11) circle (4pt);
\filldraw[black] (-5,11) circle (4pt);
\filldraw[black] (-5,8) circle (4pt);
\filldraw[black] (-4,8) circle (4pt);
\filldraw[black] (-4,6) circle (4pt);
\filldraw[black] (-3,6) circle (4pt);
\filldraw[black] (-3,4) circle (4pt);
\filldraw[black] (-1,4) circle (4pt);
\filldraw[black] (-1,3) circle (4pt);
\filldraw[black] (1,3) circle (4pt);
\filldraw[black] (1,2) circle (4pt);
\filldraw[black] (4,2) circle (4pt);
\filldraw[black] (4,1) circle (4pt);
\filldraw[black] (7,1) circle (4pt);
\filldraw[black] (7,0) circle (4pt);
\draw (-7,14)--(-6,14)--(-6,11)--(-5,11)--(-5,8)--(-4,8)--(-4,6)--(-3,6)--(-3,4)--(-1,4)--(-1,3)--(1,3)--(1,2)--(4,2)--(4,1)--(7,1)--(7,0);
\filldraw[red] (14,14) circle (4pt) node[anchor=west,scale=0.7] {$(0)$};
\filldraw[red] (14,10) circle (4pt);
\filldraw[red] (14,6) circle (4pt)node[anchor=west,scale=0.7]{$(-4)$};
\filldraw[red] (14,3) circle (4pt) node[anchor=west,scale=0.7]{$(-6)$};
\filldraw[red] (14,0) circle (4pt) node[anchor=west,scale=0.7] {$(-8)$};
\filldraw[red] (14,-1) circle (4pt) node[anchor=west,scale=0.7] {$(-10)$};
\filldraw[red] (14,-3) circle (4pt) node[anchor=west,scale=0.7] {$(-12)$};
\filldraw[red] (14,-4) circle (4pt);
\filldraw[red] (14,-6) circle (4pt);
\filldraw[red] (14,-7) circle (4pt);
\filldraw[red] (14,-8) circle (4pt);
\filldraw[red] (14,-10) circle (4pt);
\filldraw[red] (14,-11) circle (4pt);
\filldraw[red] (14,-12) circle (4pt);
\filldraw[red] (14,-14) circle (4pt);
\filldraw[red] (14,-15) circle (4pt);
\draw[red,dotted,->] (14,-16)--(14,-18);
\filldraw[black] (25,2) circle (4pt) node[anchor=west,scale=0.7]{$(4)$};
\filldraw[black] (25,1) circle (4pt);
\filldraw[black] (25,0) circle (4pt);
\filldraw[black] (25,-1) circle (4pt);
\filldraw[black] (25,-2) circle (4pt) node[anchor=west,scale=0.7]{$(0)$};
\filldraw[black] (24,-2) circle (4pt);
\filldraw[black] (24,-1) circle (4pt);
\filldraw[black] (23,-1) circle (4pt);
\filldraw[black] (23,0) circle (4pt);
\filldraw[black] (24,1) circle (4pt);
\filldraw[black] (24,0) circle (4pt);
\filldraw[black] (24,-3) circle (4pt);
\filldraw[black] (23,-3) circle (4pt);
\filldraw[black] (23,-2) circle (4pt);
\filldraw[black] (22,-2) circle (4pt);
\filldraw[black] (22,-1) circle (4pt);
\draw (24,-3)--(23,-3)--(23,-2)--(22,-2)--(22,-1);
\draw (25,-2)--(24,-2)--(24,-1)--(23,-1)--(23,0);
\draw (25,-1)--(25,0)--(24,0)--(24,1);
\draw (25,1)--(25,2);
\draw[dotted,->] (22,-3)--(21,-4);
\filldraw[magenta] (34,-2) circle (4pt) node[anchor=west,scale=0.7]{$(-8)$};
\filldraw[black] (33,-2) circle (4pt);
\filldraw[magenta] (33,-1) circle (4pt);
\filldraw[black] (32,-2) circle (4pt);
\filldraw[magenta] (32,-1) circle (4pt);
\draw (34,-2)--(33,-2)--(33,-1)--(32,-2)--(32,-1);
\filldraw[cyan] (35,1) circle (4pt) node[anchor=west,scale=0.7]{$(-6)$};
\filldraw[black] (34,-1) circle (4pt);
\filldraw[cyan] (34,0) circle (4pt);
\filldraw[cyan] (33,1) circle (4pt);
\filldraw[black] (33,0) circle (4pt);
\draw (35,1)--(34,-1)--(34,0)--(33,0)--(33,1);
\draw[dashed,magenta] (36,0.5)--(30,0.5);
\draw[dashed,cyan] (36,1.5)--(30,1.5);
\filldraw[black] (33,-3) circle (4pt) node[anchor=west,scale=0.7]{$(-10)$};
\filldraw[black] (32,-3) circle (4pt);
\filldraw[black] (32,-4) circle (4pt);
\filldraw[black] (31,-2) circle (4pt);
\filldraw[black] (31,-3) circle (4pt);
\draw (33,-3)--(32,-4)--(32,-3)--(31,-3)--(31,-2);
\filldraw[black] (36,4) circle (4pt) node[anchor=west,scale=0.7]{$(-4)$};
\filldraw[black] (35,2) circle (4pt);
\filldraw[black] (35,3) circle (4pt);
\filldraw[black] (34,2) circle (4pt);
\filldraw[black] (34,1) circle (4pt);
\draw (36,4)--(35,2)--(35,3)--(34,1)--(34,2);
\filldraw[black] (36,16) circle (4pt) node[anchor=west,scale=0.7]{$(4)$};
\draw[dotted] (35,4)--(36,14);
\draw[dotted,->] (31,-4)--(30,-6);
\filldraw[white] (0,0) circle node[black,anchor=north]{$C_2$};
\draw[->] (9,0)--(13,0);
\node[anchor=south] at (11,0) {reduce};
\node at (20,0) {$\otimes$};
\node at (28,0) {$\cong$};
\node at (24,-6) {$C^-$};
\node at (34,-6) {$\underline{C_2}\otimes C^-$};
\node at (12,-10) {$\underline{C_2}$};
\node at (39,0) {$0$};
\node at (39,3) {$3$};
\node at (39,6) {$6$};
\node at (39,9) {$9$};
\node at (39,12) {$12$};
\node at (39,15) {$15$};
\node at (39,-3) {$-3$};
\node at (39,-6) {$-6$};
\node at (39,-9) {$-9$};
\node at (39,-12) {$-12$};
\node at (39,-15) {$-15$};
\node at (39,16) {$16$};
\node at (39,17.5) {$\underline{A}$};
\end{tikzpicture}
\caption{To the left is a generating set for the complex $C_2$. Reducing as in Figure \ref{lspacereduced} gives the red complex $\underline{C_2}$. Tensoring this with $C^-$ yields the result on the right (some of the complex which is not relevant to the computations at hand has been dotted out). The homological gradings of selected generators are given in parentheses. The element colored magenta -- the highest grading generator of homology below the magenta line -- demonstrates that $V_0=-\frac{1}{2}(-8)=4$, while the element colored blue demonstrates that $V_1=-\frac{1}{2}(-6)=3$.}\label{figure:vicomp}
\end{figure}
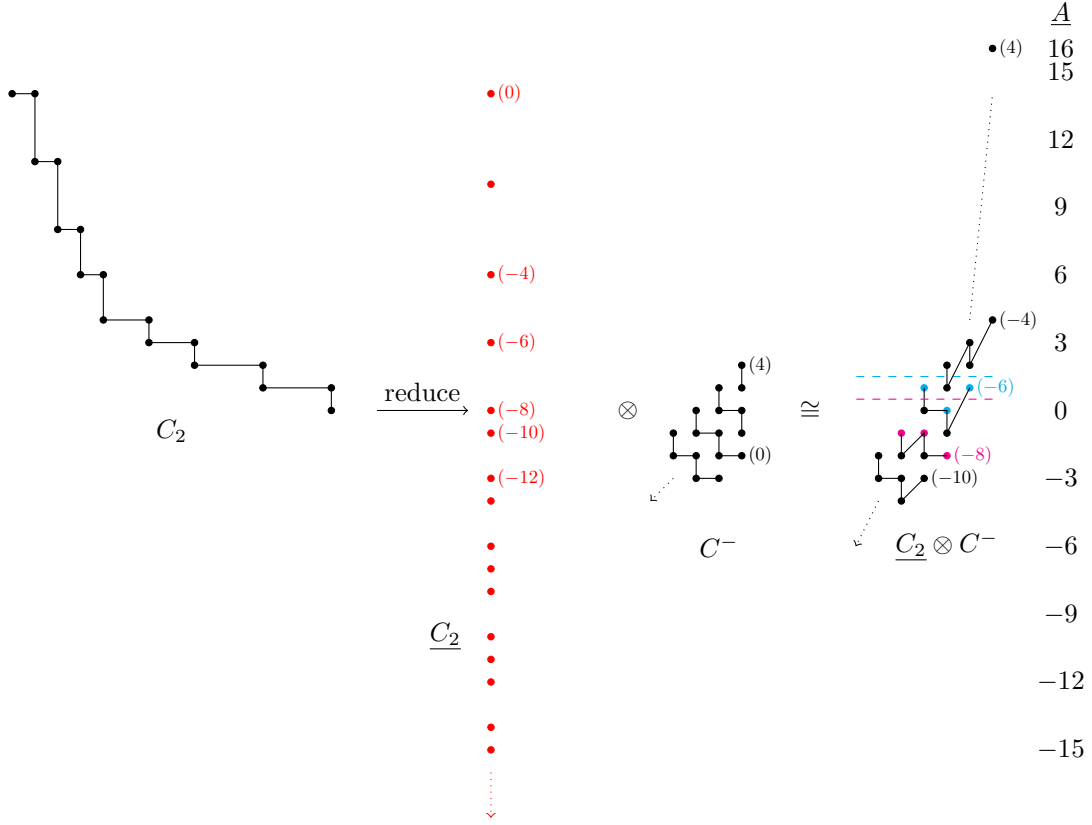

Now we wish to compute $V_0$ and $V_1$ for the complex $C_n\otimes C^-$, which is easily done using the methods in the proof of \cite[Theorem 4.7]{Krcatovich:2015-1}. To that end, let $\underline{C_n}$ be the reduced complex obtained from $C_n$ as described in Section \ref{section:reduced}. As a graded complex, \[ \underline{C_n} \cong \mathbb{F}[U]_{(0)}, \] but the filtration depends on the staircase; see Remark \ref{remark:reducedfiltration}. In the case at hand, the filtration level $A$ is given by 
\begin{equation}\label{eq:reducedfilt}
A(U^i a) = 
\begin{cases}
8n-2 - 4i & 0\leq i \leq 2n-2 \\
3 & i = 2n-1 \\
0 & i = 2n \\
-1 & i= 2n+1 \\ 
-3 & i=2n+2 \\
-4 & i=2n+3 \\
2n-i-2-\left\lfloor \frac{i-2n-4}{3} \right\rfloor & 2n+4 \leq i \leq 8n-2\\
-i & i \geq 8n-2,
\end{cases}
\end{equation}
where $a$ is the generator. Recall that $U$ drops grading by 2, so that $M(U^ia)=-2i$.

Now we consider the complex $\underline{C_n} \otimes C^-.$ This complex is generated as an $\mathbb{F}[U]$-module by $ax_j$, for $1\leq j \leq5$. It is still filtered, but multiplication by $U$ is not homogeneous with respect to the filtration, precisely because it is not in $\underline{C_n}$. If we let $A$ be the filtration on $\underline{C_n}$ as given in \eqref{eq:reducedfilt}, $A^-$ be the filtration on $C^-$, and $A^{\otimes}$ be the filtration on the product, then 
\begin{equation}\label{eq:tensorfilt}
A^{\otimes}(U^ia x_j) = A(U^ia) + A^-(x_j).
\end{equation} 
Note that, in general, to get a filtration on a tensor product, we would set \[ A^{\otimes}(U^iax_j) = \min_{0\leq k\leq i} \{ A(U^ka)+A^-(U^{i-k}x_j) \}.\] Because $U$ lowers $A^-$ by 1 and $U$ lowers $A$ by {\it at least} 1, this is given by \eqref{eq:tensorfilt}.

\begin{remark}\label{remark:tensorgradings}
The homogeneous elements in $\underline{C_n} \otimes C^-$ with homological grading $-2i$ are precisely $U^{i+2}a x_1, U^{i+1}a x_3$, and $U^{i}a x_5$. Further, $H_{-2i}(\underline{C_n} \otimes C^-)$  is generated by the sum of these elements. It follows from this and \eqref{eq:tensorfilt} that the filtration level of the generator of $H_{-2i}$ is given by \[ \max \{ 2+A(U^{i+2}a), A(U^{i+1}a), -2+A(U^{i}a) \}.\]
\end{remark}

Now we wish to determine the $V_k$'s, recalling that $\underline{C_n}\otimes C^-$ is $\nu^+$-equivalent to $2J_n$. Note that \[  \max \{ 2+A(U^{2n+2}a), A(U^{2n+1}a), -2+A(U^{2n}a)     \} = -1,\] while \[  \max \{ 2+A(U^{2n+1}a), A(U^{2n}a), -2+A(U^{2n-1}a)     \} = 1.\] That is, the filtration level of the generator of $H_{-4n}$ is -1, and the filtration level of the generator of $H_{-4n+2}$ is 1. Together with \eqref{equation:Vifromreduced}, these give that $V_0(2J_n) = -\frac{1}{2}(-4n) = 2n$, and $V_1(2J_n) \leq 2n-1$. By Proposition \ref{proposition:Videcreasing}, $V_1$ and $V_0$ can differ by at most 1, so it follows that $V_1(K)=2n-1$. 
\end{proof}

%If we define \[ h(i):= \min \{ k \mid (\underline{C_n} \otimes C^- )\{ j\leq k\} \textrm{ has a non-torsion generator of homology in grading }-2i \},\] then in particular, following \eqref{equation:Vifromreduced},  \[ V_0(2J_n) = V_0 (\underline{C_n}\otimes C^-) = \min \{ i \mid h(i) \leq 0 \}, \] and \[ V_1(2J_n) = V_1 (\underline{C_n}\otimes C^-) = \min \{ i \mid h(i) \leq 1 \}. \]

%Now, using \eqref{eq:reducedfilt}, we compute that \[ h(2n) = \max \{ 2+A(U^{2n+2}a), A(U^{2n+1}a), -2+A(U^{2n}a)     \} = -1,\] while \[ h(2n-1) = \max \{ 2+A(U^{2n+1}a), A(U^{2n}a), -2+A(U^{2n-1}a)     \} = 1.\] Together, these give that $V_0(2J_n) = 2n$ and $V_1(2J_n)\leq2n-1$. By Proposition \ref{proposition:Videcreasing}, $V_1$ and $V_0$ can differ by at most 1, so it follows that $V_1(K)=2n-1$. 

%===============================================================================================================
\section{Proof of Theorem \ref{theorem:A}}\label{section:maintheorem}
The goal of this section is to prove Theorem \ref{theorem:A}. Let $n$ be a positive integer and $L_n$ be the link in Figure \ref{figure:mainexample}. In the introduction section, we observed that $L_n$ has unknotted components and $L_n$ is topologically concordant to the Hopf link for any integer $n$. It remains to prove that $L_n$ is not smoothly concordant to  any $2$-component link $J$ with trivial Alexander polynomial, and $L_n$ and $L_m$ are not smoothly concordant when $n\neq m$. Therefore, the following theorems imply Theorem \ref{theorem:A}.
\begin{figure}[htb]
\centering
\includegraphics[width=2in]{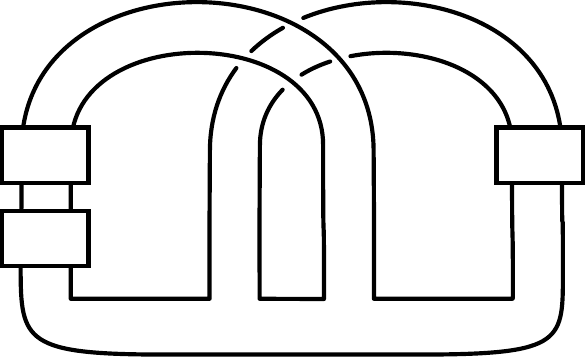}
\put(-135,46.5){$1$}
\put(-139,26){$J_n$}
\put(-17.5,47.3){$-2$}
\caption{A disk-band form of a genus $1$ Seifert surface of $K_{L_n}$.}
\label{figure:bandform}
\end{figure}

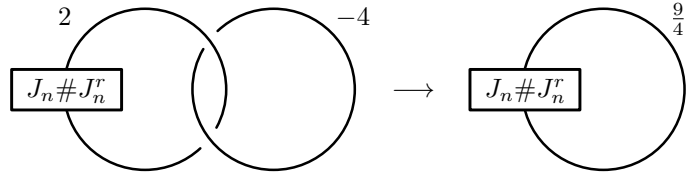
\begin{figure}[htb]
\centering
\includegraphics[width=3.5in]{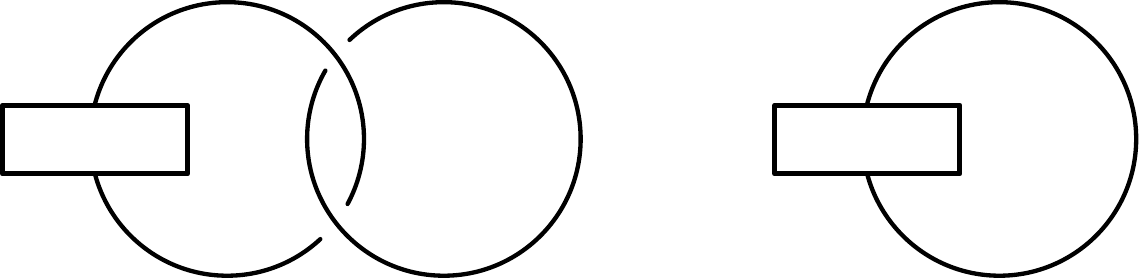}
\put(-247,28){$J_n\#J_n^r$}
\put(-75.2,28){$J_n\#J_n^r$}
\put(-235,55){$2$}
\put(-131,55){$-4$}
\put(-6,55){$\frac{9}{4}$}
\put(-110,28){$\longrightarrow$}
\caption{The $2$-fold branched cover of $K_{L_n}$. The second figure is obtained by the slam-dunk move from the first figure.}
\label{figure:cover}
\end{figure}
\begin{theorem}\label{theorem:Arepeat}For any positive integer $n$, the link $L_n$ is not smoothly concordant to any $2$-component link $J$ with trivial Alexander polynomial.
\end{theorem}
\begin{proof}Throughout the proof, we continue to use the notations given in the statement of Theorem \ref{theorem:mainobstruction}. Let $Z_n=\Sigma_{K_{L_n}}$ be the double cover of $Y_{L_n}$ branched along $K_{L_n}$. Suppose that $L_n$ is concordant to a link $J$ with trivial Alexander polynomial.  By Theorem \ref{theorem:mainobstruction}, there is a metabolizer $M$ for the linking form $\lambda_{Z_n}$ of $Z_n$ such that $\overline{d}(Z_n,\mathfrak{s}_m)=0$ for all $m$ in $M$.

Figure \ref{figure:bandform} gives a disk-band form of a genus 1 Seifert surface of $K_{L_n}$. Using the Akbulut-Kirby method \cite{Akbulut-Kirby:1979-1}, we can draw a surgery diagram of $Z_n$ as given in Figure \ref{figure:cover}. Let $\mu$ be a meridian of the second figure of Figure \ref{figure:cover}. $H_1(Z_n)$ is generated by $\mu$ and $\lambda_{Z_n}(\mu,\mu)=-\frac{4}{9}\in \Q/\Z$. The linking form $\lambda_{Z_n}$ has unique metabolizer $M$ generated by $3\mu$. 

To obtain a contradiction, we show that $\overline{d}(Z_n,\mathfrak{s}_{3\mu})\neq 0$. Since $Z_n=S^3_{9/4}(J_n\#J_n^r)$, the Spin structure $\mathfrak{s}_0$ corresponds to the Spin$^c$ structure $\mathfrak{t}_6$ by Remark~\ref{remark:spinclabel}. We now observe that $\mathfrak{s}_{3\mu}=\mathfrak{t}_0$. By the definition of $\mathfrak{s}_{3\mu}$ and Remark \ref{remark:spinclabel},
\[\mathfrak{s}_{3\mu}=\mathfrak{s}_0+3\hat{\mu}= \mathfrak{t}_6+3\hat{\mu}=\mathfrak{t}_{6}+3\cdot 7\hat{\mu}=\mathfrak{t}_{0}.\]
Further, using the recursive formula from \cite[Section $4$]{Ozsvath-Szabo:2003-2}, it can be easily verified that $d(S^3_{9/4}(U),\mathfrak{t}_0)=d(S^3_{9/4}(U),\mathfrak{t}_6)=0$.

By Proposition \ref{proposition:Viformula} and Theorem \ref{theorem:C}, we have
\begin{align*}
\overline{d}(Z_n,\mathfrak{s}_{3\mu})&=d(Z_n,\mathfrak{s}_{3\mu})-d(Z_n,\mathfrak{s}_0)\\
&= d(S^3_{9/4}(J_n\# J_n^r),\mathfrak{t}_0)-d(S^3_{9/4}(J_n\# J_n^r),\mathfrak{t}_6)\\
&=d(S^3_{9/4}(U),\mathfrak{t}_0)-2V_0(J_n\#J_n^r)-d(S^3_{9/4}(U),\mathfrak{t}_6)+2V_1(J_n\#J_n^r)\\
&=2V_1(J_n\# J_n^r)-2V_0(J_n\#J_n^r)\\
&=2V_1(2J_n)-2V_0(2J_n)=4n-2-4n=-2,
\end{align*}
where in the last line we use the fact that $CFK^\infty(K^r)\cong CFK^\infty(K)$ \cite[Proposition 3.9]{Ozsvath-Szabo:2004-1}.
\end{proof}

\begin{theorem}If $n$ and $m$ are distinct positive integers, then the links $L_n$ and $L_m$ are not smoothly concordant.
\end{theorem}
\begin{proof}We continue to use the notations used in Theorem \ref{theorem:Arepeat}. Suppose that $L_n$ and $L_m$ are smoothly concordant. The proof of Theorem \ref{theorem:mainobstruction} shows that there is a $\Z_2$-homology cobordism $W$ between $Z_n$ and $Z_m$. It follows that $d(Z_n,\mathfrak{s}_0)=d(Z_m,\mathfrak{s}_0)$. By the proof of Theorem \ref{theorem:Arepeat}, we have that $d(Z_n,\mathfrak{s}_0)=4n$ and $d(Z_m,\mathfrak{s}_0)=4m$ since $n,m\geq 1$. It follows that $n=m$. This completes the proof.
\end{proof}

\bibliographystyle{amsalpha}
\renewcommand{\MR}[1]{}
\bibliography{research}
\end{document}